\documentclass[11pt]{article}

\usepackage{graphicx}
\usepackage{multicol}
\usepackage{float}
\usepackage{url}
\usepackage{caption}
\usepackage{subcaption}
\usepackage{hyperref}
\hypersetup{hidelinks}
\usepackage[english]{babel}
\usepackage{array}
\usepackage{multirow}
\usepackage{amsmath,amssymb,amsfonts,amsthm}
\usepackage{amsthm}
\usepackage{mathrsfs}
\usepackage[title]{appendix}
\usepackage{xcolor}
\usepackage{colortbl}
\usepackage{textcomp}
\usepackage{manyfoot}
\usepackage{booktabs}
\usepackage{algcompatible}
\usepackage{algpseudocode}
\usepackage{listings}
\usepackage{color,soul}
\usepackage{longtable}
\usepackage{bbm}
\usepackage{comment}

\newtheorem{theorem}{Theorem}
\newtheorem{proposition}[theorem]{Proposition}
\newtheorem{assumption}{Assumption}

\newtheorem{definition}{Definition}

\newcolumntype{P}[1]{>{\centering\arraybackslash}p{#1}}
\newcommand{\norm}[1]{\left\lVert#1\right\rVert}

\newcommand{\cX}{{\mathcal X}}
\newcommand{\cT}{{\mathcal T}}
\newcommand{\cJ}{{\mathcal J}}
\newcommand{\cF}{{\mathcal F}}

\newcommand{\cM}{{\mathcal M}}

\newcommand{\cC}{{\mathcal C}}

\newcommand{\cA}{{\mathcal A}}

\newcommand{\cB}{{\mathcal B}}
\newcommand{\cD}{{\mathcal D}}

\newcommand{\bea}{\begin{eqnarray*}}
\newcommand{\eea}{\end{eqnarray*}}

\newcommand{\inprod}[2]{\langle #1 , #2 \rangle }

\newcommand{\bc}{\begin{center}}
\newcommand{\ec}{\end{center}}

\newcommand{\K}{\mathbb K}
\newcommand{\R}{\mathbb R}
\newcommand{\N}{\mathbb N}

\renewcommand{\L}{\mathbb L}
\renewcommand{\S}{\mathbb S}

\newcommand{\ee}{\end{equation}}
\newcommand{\beaa}{\begin{eqnarray*}}
\newcommand{\eeaa}{\end{eqnarray*}}
\newcommand{\ben}{\begin{enumerate}}
\newcommand{\een}{\end{enumerate}}
\newcommand{\db}{\hspace*{\fill}{\zapf o}}

\newcommand{\bpn}{\begin{proposition}\twlsf}
\newcommand{\epn}{\db\end{proposition}}
\newcommand{\bdm}{\begin{displaymath}}
\newcommand{\edm}{\end{displaymath}}
\newcommand{\ea}{\end{array}}

\def\texitem#1{\par\smallskip\noindent\hangindent 25pt
               \hbox to 25pt {\hss #1 ~}\ignorespaces}

\def\cE{{\mathcal E}}
\def\cK{{\mathcal K}}
\def\cG{{\mathcal G}} \def\cV{{\mathcal V}}
\def\cO{{\mathcal O}} \def\cL{{\mathcal L}}
 \def\cY{{\mathcal Y}}
\def\hX{\widehat{X}} \def\hy{\widehat{y}}
\def\hZ{\widehat{Z}}

\definecolor{orgcolor}{RGB}{0,128,96}
\newcommand{\best}[1]{\cellcolor{lightgray}#1}

\usepackage[linesnumbered,ruled,vlined]{algorithm2e}
\SetAlgorithmName{Algorithm}{Algorithm}{Algorithm}

\IncMargin{0.5em}
\SetCommentSty{textnormal}
\SetNlSty{}{}{:}
\SetAlgoNlRelativeSize{0}
\SetKwInput{KwGlobal}{Global}
\SetKwInput{KwPrecondition}{Precondition}
\SetKwProg{Proc}{Procedure}{:}{}
\SetKwProg{Func}{Function}{:}{}
\SetKw{And}{and}
\SetKw{Or}{or}
\SetKw{To}{to}
\SetKw{DownTo}{downto}
\SetKw{Break}{break}
\SetKw{Continue}{continue}
\SetKw{SuchThat}{\textit{s.t}}
\SetKw{WithRespectTo}{\textit{wrt}}
\SetKw{Iff}{\textit{iff.}}
\SetKw{MaxOf}{\textit{max of}}
\SetKw{MinOf}{\textit{min of}}
\SetKwBlock{Match}{match}{}{}
\newlength{\commentWidth}
\let\oldnl\nl
\newcommand{\nonl}{\renewcommand{\nl}{\let\nl\oldnl}}

\def\SGSadmm{{\sc sGS-PADMM }}
\def\decompadmm{{\sc McADMM }}
\def\SNCGadmm{{\sc SNCG-PADMM }}
\def\SDPNALplus{{\sc Sdpnal+ }}
\def\lam{\lambda}

\def\grad{\nabla}

\def\sig{\sigma}
\def\diag#1{\mbox{diag}(#1)}

\renewcommand{\xi}{w}

\begin{document}

\title{McADMM: A Multi-Clique Augmented Lagrangian-Based Algorithm for Large-Scale Sparse SDPs with Bound Constraints}

 \author{Kristo Nugraha Lian\thanks{Department of Mathematics, National University of Singapore, Singapore ({\tt kristo@u.nus.edu})}, \;
         Nehal Ahmed Shaikh\thanks{Independent researcher, Hyderabad, Sindh, Pakistan (\texttt{nehalahmedshaikh@gmail.com})}, \;
         Di Hou\thanks{(Corresponding author) Department of Mathematics, National University of Singapore, Singapore ({\tt dihou@u.nus.edu})}, \;
         Xingyu Xie\thanks{Department of Mathematics, National University of Singapore, Singapore ({\tt xyxie@nus.edu.sg})}, \;
         Kim-Chuan Toh\thanks{Department of Mathematics, National University of Singapore, Singapore ({\tt mattohkc@nus.edu.sg})}}

\maketitle

\begin{abstract}
    \SGSadmm \; \cite{STY2015,LST2015,CST2015} is a powerful and versatile class of convergent multi-block ADMM solvers for implementations on moderate-sized linear semidefinite programming (SDP) problems. In this paper, we further enhance this class of algorithms for solving SDP problems by proposing a new multi-clique decomposition approach, allowing substantial improvements in applications on large-scale sparse SDPs (e.g., where $n > 1000$) with conducive aggregate sparsity patterns. Our SDP decomposition strategy mainly {aims to reduce the estimated PSD projection cost after decomposition}, in contrast to common decomposition algorithms that are encumbered with minimizing the overlaps between cliques. This feature is made possible by our novel \textit{linear-space projection} approach that is capable of efficiently processing a large number of overlap constraints via simple averaging steps. For the numerical experiments, we demonstrate the performance of our solver -- named McADMM for Multi-clique ADMM -- on a number of large-scale SDP instances that arise from relaxations of some important quadratically constrained quadratic programming (QCQP) problems. The performance of McADMM is contrasted against other state-of-the-art decomposition-based solvers as well as the non-decomposed \SGSadmm to highlight our key contributions. We additionally develop a GPU implementation of McADMM and demonstrate that it can substantially accelerate the decomposed solver.
\end{abstract}


\section{Introduction} We consider the following linear SDP with bound constraints:
\begin{align}\label{eq.1}
    \min_{X}	\hspace{0.1cm}   & \inprod{C}{X}  \nonumber\\[5pt]
    \text{s.t.}     \hspace{0.1cm}  &	 \cA (X) = b, \quad X\in \S^{n}_+, \quad X\in \N^{n},
\end{align}
where $C \in \S^{n}$, $b \in \R^m$ and $\cA (X) =\left[\inprod{A_1}{X}, \inprod{A_2}{X}, \ldots, \inprod{A_m}{X} \right]^\top$ is a given linear map on the matrix variable $X \in \S^{n}$. Furthermore, $\S^{n}_+$ denotes the cone of positive semidefinite (PSD) matrices and $\N^{n}$ is a set corresponding to the bound constraints on $X$. That is,  $\N^{n} = \{ X \in \S^{n} \; : \; L \leq X \leq U \}$ for some given {symmetric lower and upper bounds} $L,U \in {(\R\cup\{-\infty,+\infty\})^{n\times n}}$. We allow $L$ and $U$ to contain $-\infty$ and $\infty$ elements, respectively, so that some elements of $X$ may not have bound constraints. Doubly non-negative SDP problems are special cases of \eqref{eq.1}, with $L = 0_{n\times n}$ and $U = \mathbf{\infty}_{n\times n}$.

Under certain conditions on its data sparsity structure, it is possible to convert \eqref{eq.1} into a multi-block SDP whose primal variables exhibit a \textit{quasi-block diagonal} structure. The term \textit{quasi-block diagonal} refers to a sparsity structure whose non-zero elements are contained in smaller, possibly non-contiguous sub-blocks that may include some overlaps -- see Figure~\ref{fig:quasi-block} for an illustration.

\begin{figure}
    \begin{center}
        \includegraphics[width = 0.31\textwidth]{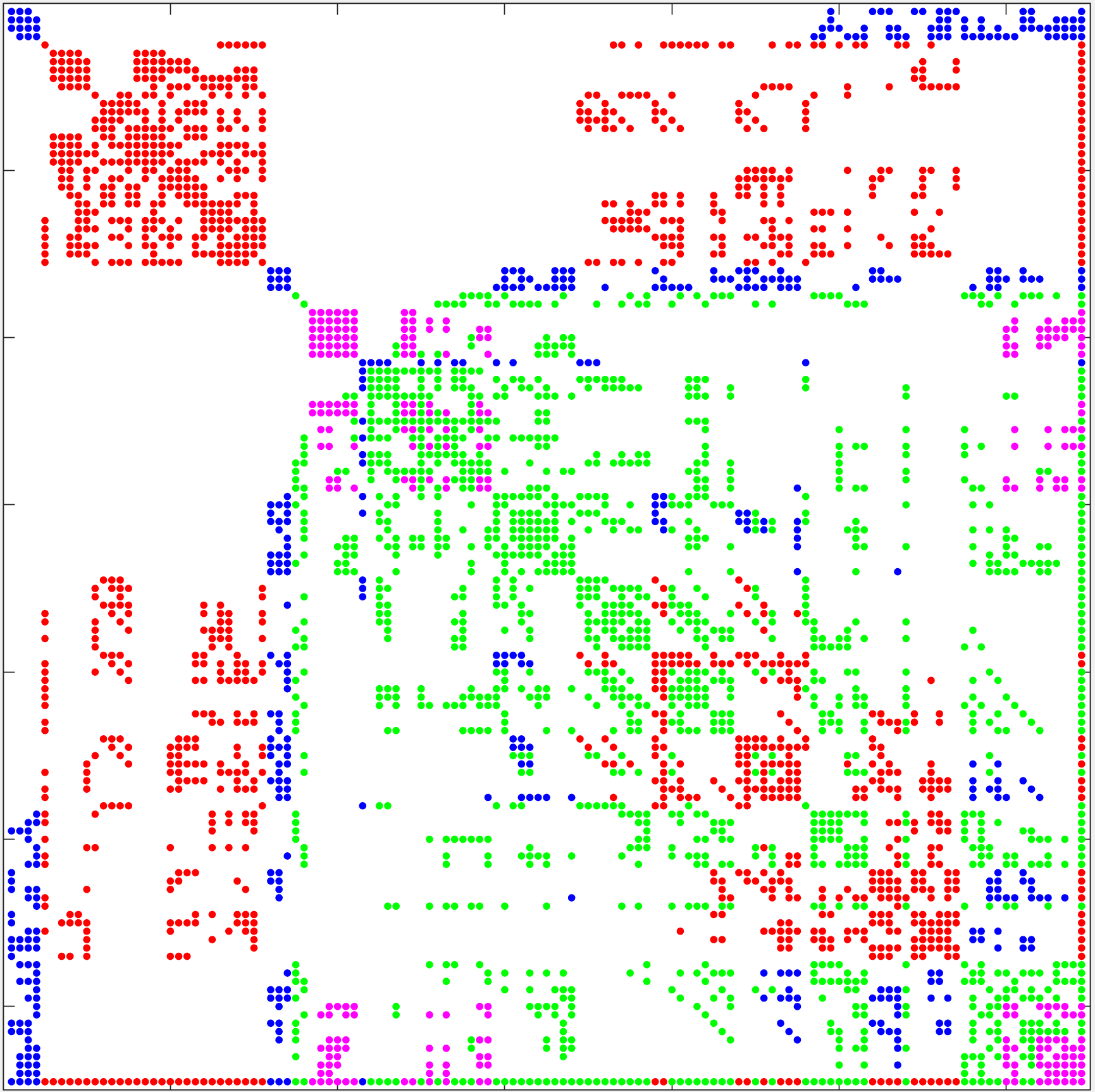}
        \caption{Illustration of a \textit{quasi-block diagonal} structure. This pattern corresponds to the Lovász theta SDP instance 1dc.128 upon \textit{chordal decomposition} and reordering as discussed in Section~\ref{sec-conversion}. The clique assignments -- based on our CMD algorithm in Section~\ref{subsec-3.2} -- are color coded.}\label{fig:quasi-block}
    \end{center}
\end{figure}

\subsection{A multi-block SDP} To facilitate a multi-block SDP whose primal variables exhibit a \textit{quasi-block diagonal} structure, we present a class of multi-block SDPs in this subsection.

Let $\cX=\cX^{(1)}\times\ldots\times\cX^{(p)}$, where each $\cX^{(l)}=\S^{n_l}$. For each $l$, we let $\N^{(l)} = \{ X^{(l)} \in \cX^{(l)} \mid L^{(l)} \leq X^{(l)} \leq U^{(l)} \}$ where $L^{(l)}$, $U^{(l)}$ are given matrix bounds. Once again, elements of the matrices $L^{(l)}$ and $U^{(l)}$ are allowed to be $-\infty$ and $\infty$ respectively to represent unconstrained entries.

Consider the following multi-block linear SDP with $p$ blocks of variables $X^{(1)},\ldots,X^{(p)}$:
\begin{align}
    \min_{X^{(1)}, \ldots, X^{(p)}} \hspace{0.1cm} &  \sum_{l=1}^p\inprod{C^{(l)}}{X^{(l)}} \nonumber \\
    \mbox{s.t.} \;&  \sum_{l=1}^p \cA^{(l)} (X^{(l)}) = b, \quad X^{(l)} \in \S^{n_l}_+, \; X^{(l)} \in \N^{(l)} \nonumber \\
    & \cL(X) = h, \quad X:=(X^{(1)},\ldots,X^{(p)}) \in \cX,
    \label{eq.2}
\end{align}
where $b\in \R^m$, $h\in \R^q$, and for each $l$,  $C^{(l)} \in \cX^{(l)}$ are given data, $\cA^{(l)} : \cX^{(l)} \to \R^m$ is a given linear map. We note that $\cL:\cX \to\R^q$ also describes a linear map, and is typically assigned to impose \textit{consistency}-type constraints for overlapping blocks in a quasi-block diagonal structure, which will be presented later. As demonstrated in the following sections, this structural distinction of linear constraints easily conforms to the multi-block structure of our algorithms, on top of promoting computational efficiency on problems admitting \textit{quasi-block diagonal} structures.

SDPs of the form \eqref{eq.2} are typically obtained via chordal (or clique) decomposition from those of the form \eqref{eq.1} -- in many important applications such as matrix completion problems (see Section \ref{sec-maxnorm}) and MAXCUT problems, chordal decomposition can be directly applied with no prior pre-processing. In other cases, such as the maximum stable set problems, the decomposition involves the prior sparsification steps outlined in Appendix~\ref{sec-theta}.

Note that the above SDP \eqref{eq.2} can be easily generalized by considering the constraints $ X^{(l)} \in \K^{(l)}$ in place of $ X^{(l)} \in \S^{n_l}_+$, where $\K^{(l)}\subset \cX^{(l)}$ may represent the cone of nonnegative vectors $\R^{n_l}_+$ or the second-order cone. Except for some adjustments in the projection steps, the structure of our main algorithms (Algorithms~\ref{alg:mcadmm-general} and~\ref{alg:mcadmm-bounded}) would remain unchanged in these cases. Thus for simplicity, our discussions in this paper would mainly be restricted to the SDP \eqref{eq.2}.

\subsection{The convergent multi-block ADMM and other solvers} To solve \eqref{eq.1} or \eqref{eq.2} with $p= 1$, a powerful class of multi-block ADMM with proven convergence (see \cite{LST2015} \cite{STY2015}, and \cite{CST2015}) -- referred here as \SGSadmm -- has been well established. One notable variant is known as \SDPNALplus \cite{YST2015} -- this algorithm involves a semismooth Newton-CG approach to solve the ADMM subproblems and exhibits excellent performance on SDPs of moderate size (say, with $n \sim 1000$) but with a large number of equality constraints (say, with $m > 10^5$). To improve the scalability of the above solvers against the matrix dimension $n$, we propose the use of clique decomposition technique introduced in \cite{FKM2000,NFF2003} to exploit any sparsity structure present in the data and to decompose the original problem \eqref{eq.1} into the form \eqref{eq.2} whose matrices are of smaller dimensions.

As discussed in Section~\ref{sec-conversion}, this decomposition possibly introduces a large number of additional linking equality constraints to decouple the overlapping blocks into independent blocks. However, we should emphasize that the additional linking constraints introduced will not create substantial computation burden to the algorithms we will design. This is in sharp contrast to interior-point method (IPM) based solvers for which the computational and memory cost will increase drastically with the number of additional linking (or consistency) constraints \cite{NFF2003}.

While the multi-block SDPs obtained from our clique decompositions (see Section~\ref{sec-clique-decomp}) generally retain the structure presented in \eqref{eq.2}, the algorithm presented in this paper can be extended to applications on an even more general model given as follows:
\begin{align}
    \min_{X^{(1)}, \ldots, X^{(p)}} \hspace{0.1cm} &  \sum_{l=1}^p\Bigl(\inprod{C^{(l)}}{X^{(l)}} + f^{(l)}(X^{(l)}) + g^{(l)}(X^{(l)})\Bigr) \nonumber \\
    \mbox{s.t.} \; &  \sum_{l=1}^p \cA^{(l)} (X^{(l)}) = b \nonumber \\
    & \cL(X) = h, \quad X:=(X^{(1)},\ldots,X^{(p)}) \in \cX,
    \label{eq.3}
\end{align}
where $f^{(l)}: \cX^{(l)} \to (-\infty,\infty]$ and $g^{(l)}: \cX^{(l)} \to (-\infty,\infty]$ are proper closed convex functions. The key requirement we need for $f^{(l)}$ and $g^{(l)}$, $l=1,\ldots,p$, is that their proximal mappings can be computed efficiently. It is easy to see that the above model \eqref{eq.3} includes the problem \eqref{eq.2} as a special case when we take $f^{(l)}(X^{(l)}) = \delta_{\S^{n_l}_+}(X^{(l)})$ and $g^{(l)}(X^{(l)}) = \delta_{\N^{(l)}}(X^{(l)})$. The indicator function $\delta_{\S^{n_l}_+}:\cX^{(l)}\to (-\infty,\infty]$ is defined by $\delta_{\S^{n_l}_+}(X^{(l)}) = 0$ if $X^{(l)}\in \S^{n_l}_+$, and $\infty$ otherwise. The function $\delta_{\N^{(l)}}(\cdot)$ is defined similarly.

For applications on SDPs of the form \eqref{eq.2}  with $\N^{(l)} = \cX^{(l)}$ for all $l$ (linear SDPs without explicit bound constraints), \cite{FKM2000} and \cite{KKMY2011} proposed to apply a primal-dual IPM to solve the multi-variable SDP derived using their clique decomposition strategy. However, as we remarked, the IPM approach proves to be expensive due to the large number of additional linking constraints introduced. Furthermore in \cite{SAV2014}, the authors proposed to apply a Douglas-Rachford splitting method on the multi-variable problem for which the key computation in each iteration is the solution of a least-squares SDP problem satisfying the first set of linear constraints $\sum_{l=1}^p \cA^{(l)} (X^{(l)}) = b$ and cone constraints $X^{(l)} \in \S^{n_l}_+$ for all $l=1,\ldots,p$ in \eqref{eq.2}. While \cite{SAV2014} specifically designed a dedicated IPM to reduce the cost of solving the least-squares SDP in each iteration, unfortunately the overall cost can still be quite high.

\subsection{Notation and preliminaries} In all the following sections, we shall discuss only SDPs of the decomposed form \eqref{eq.2} or \eqref{eq.3}. {We denote the $i$th standard basis vector by $e_i$.} Let $[n]:=\{1,2,\ldots,n\}$ for any positive integer $n$. We {use the following product-space inner product and norm on $\cX$}
\[
\langle X,Y\rangle_{\cX}
:=\sum_{l=1}^p\langle X^{(l)},Y^{(l)}\rangle_F,
\qquad
\norm{X}_{\cX}^2
:=\sum_{l=1}^p\norm{X^{(l)}}_F^2.
\]
For notational simplicity, it is useful to express these multi-variable SDPs in the standard form by defining $ \mathscr{C}=(C^{(1)},\ldots,C^{(p)})$ and define for $X =(X^{(1)},\ldots,X^{(p)}) \in \cX$,
\begin{eqnarray}
    \begin{array}{l}
        \mathscr{A}(X) = \sum_{l=1}^p \cA^{(l)}(X^{(l)}),  \quad
        \mathscr{F}(X) = f^{(1)} (X^{(1)}) + \cdots +  f^{(p)}(X^{(p)})
        \\[5pt]
        \mathscr{G}(X) = g^{(1)} (X^{(1)}) + \cdots +  g^{(p)}(X^{(p)}).
    \end{array}
    \label{eq.4}
\end{eqnarray}
With this, \eqref{eq.3} can be written compactly as:
\begin{align}
    \min_X \hspace{0.1cm} & \inprod{\mathscr{C}}{X} + \mathscr{F}(X) + \mathscr{G}(X) \nonumber\\
    \mbox{s.t.} \;& \mathscr{A}(X) = b, \quad
    \cL (X) = h, \quad X \in \cX.
    \label{eq.5}
\end{align}
Its dual problem (ignoring the minus sign in front of the minimization) is given by
\begin{align}
    \min_{y, \xi, S, Z} \hspace{0.1cm} & \inprod{-b}{y} + \inprod{-h}{\xi} + \mathscr{F}^*(-S) + \mathscr{G}^*(-Z)
    \nonumber \\
    \mbox{s.t.} \;& \mathscr{A}^* y + \cL^* \xi + S + Z = \mathscr{C},
    \label{eq.6}
\end{align}
where $S = (S^{(1)},\ldots,S^{(p)})$, $Z = (Z^{(1)},\ldots,Z^{(p)})$ and
\begin{eqnarray}
    \mathscr{F}^*(S) & :=& (f^{(1)})^*(S^{(1)}) + \cdots +  (f^{(p)})^*(S^{(p)})  ,
    \\[5pt]
    \mathscr{G}^*(Z) & :=& (g^{(1)})^*(Z^{(1)}) + \cdots +  (g^{(p)})^*(Z^{(p)}).
\end{eqnarray}

Throughout this paper, we will assume that the problems \eqref{eq.5} and \eqref{eq.6} have a solution satisfying the KKT optimality conditions:
\begin{eqnarray}
    \begin{array}{lcl}
        0 &=& \mathscr{A}^* y + \cL^*\xi + S + Z -\mathscr{C} \\[5pt]
        0 &= & \mathscr{A} (X) -b, \quad \quad \; \; \; 0 = \cL (X) -h \\[5pt]
        0 &\in &\partial\mathscr{F}^* (-S) -X, \quad 0 \;\in\; \partial \mathscr{G}^* (-Z) -X.
    \end{array}
    \label{eq.9}
\end{eqnarray}

In later part of this paper, we will make use of the Moreau decomposition theorem for a proper closed convex function, which we present next. Let $\cY$ be a finite dimensional real inner product space equipped with the induced norm, and $f:\cY \to (-\infty,\infty]$ be a proper closed convex function. The proximal map and the Moreau-Yosida regularization associated with $f$ are defined by
\begin{eqnarray*}
    {\rm Prox}_f (x) = \mbox{argmin}_{z\in \cY} \Big\{ f(z) + \frac{1}{2}\norm{z-x}^2 \Big\},\quad
    M_f(x) =  \mbox{min}_{z\in \cY} \Big\{ f(z) + \frac{1}{2}\norm{z-x}^2 \Big\}, \quad
    x\in \cY.
\end{eqnarray*}
It is known that $M_f(\cdot)$ is continuously differentiable on $\cY$, with
$$
\grad M_{ f}(x) = x- {\rm Prox}_{ f}(x).
$$
We have the following fundamental identities due to Moreau \cite{Moreau1962}: for any $x\in \cY$, it holds that
\begin{eqnarray}
    {\rm Prox}_{\sig f}(x) + \sig {\rm Prox}_{\sig^{-1}f^* }(\sig^{-1}x) &=& x,
    \label{eq-Moreau}
    \\[5pt]
    M_{\sig f}(x) + \sig^2 M_{\sig^{-1}f^*}(\sig^{-1} x) &=& \frac{1}{2}\norm{x}^2,
    \label{eq.10}
\end{eqnarray}
where $\sig >0$ is a given parameter, and $f^*$ is the Legendre–Fenchel conjugate of $f$.

\subsection{Contributions} The main contributions of this paper are as follows:
\begin{enumerate}
\item Our solver can process single or multi-block SDPs of the form \eqref{eq.3}, which is a significant generalization to those considered in the literature. In particular, any clique-decomposed SDPs with bound constraints can be solved efficiently and with improved scalability compared to other first-order augmented Lagrangian based methods and most notably, IPM based ones. In our solver, we fully exploit clique decomposition by decomposing both the objective data and all constraints. This contrasts with the more common approach of decomposing only the \(\mathbb{S}^n_+\) constraint, as proposed, for example, in \cite{ZFP2020}.

\item In most if not all decomposition based SDP solvers, the previously discussed linking constraints that ensure variable consistency across different blocks are typically encoded as linear \textit{consistency} constraints (see, e.g., \cite{FKS2017} and \cite{ZFP2020}). In section \ref{sec-linproj}, we propose a \textit{linear-space projection} strategy to efficiently impose these constraints such that their updates in the algorithm come down to a simple averaging operation across all blocks. Furthermore, we demonstrate that such an operation can be done concurrently with the updating step pertaining to the bound constraint (projection to $\N^n$). In this regard, our approach intelligently overcomes the major challenge that is the main hindrance of decomposition-based SDP solvers.

\item Another benefit of our \textit{linear-space projection} method is reflected in the clique decomposition stage. Here, we propose a heuristic for maximal clique distribution upon chordal decomposition that differs from the commonly used heuristics (e.g. SparseCoLO \cite{KKMY2011}). More precisely, we design a decomposition algorithm -- called CMD -- {that greedily merges cliques to reduce the estimated PSD projection cost}. In contrast, older heuristics such as those in SparseCoLO are more restricted as they must keep the number of clique overlaps in check to avoid excessive computational burden incurred by the linking constraints.
\item Our solver, called McADMM, is versatile and capable of accommodating \textit{any} type of decomposition, including non-chordal ones such as the block-factor-width decompositions (see e.g., \cite{ZSP2019} and \cite{ZFP2021}). Indeed, the structure of our main algorithm (Algorithm~\ref{alg:mcadmm-general}) is independent of the decomposition strategies developed in Section~\ref{sec-clique-decomp}.
\item We demonstrate that the multi-clique structure of McADMM is well suited to GPU acceleration. Our implementation parallelizes the overlap projection and the clique-wise PSD projections on GPUs. Computational results show substantial acceleration on sufficiently large instances, with several-fold median CPU-to-GPU speedups across the tested problem classes.
\end{enumerate}

\subsection{Organization} \label{subsec-organization} This paper is organized as follows. Section~\ref{sec-conversion} introduces the aggregate sparsity pattern and converts a single-block SDP into a multi-block structured SDP, including the consistency constraints induced by clique overlaps. Section~\ref{sec-clique-decomp} develops clique decomposition strategies for arrowhead-patterned and general sparse SDPs, including the proposed CMD heuristic. Section~\ref{sec-computational} presents the implementation details for overlap compensation and the linear-space projection. Section~\ref{sec-ADMM} develops McADMM for general and bounded multi-block SDPs. Section~\ref{sec-numerical} reports the numerical experiments, including comparisons with existing solvers and the CPU--GPU results. Finally, Section~\ref{sec-summary} concludes the paper.

\section{Conversion from single to multi-block structured SDP} \label{sec-conversion} Non-trivial conversions such as the one depicted in Figure~\ref{fig:quasi-block} are attainable provided \eqref{eq.1} has a conducive \textit{aggregate sparsity pattern}. First, we shall provide a more precise definition of the aggregate pattern based on the SDP parameters $C, A_1,\ldots,A_m\in \S^n$ in \eqref{eq.1}, together with the bounds $L, U \in (\R\cup\{-\infty,\infty\})^{n \times n}$. The first description can be presented in terms of the following matrix:
\begin{align}\label{eq.12}
    M^{\text{par}} = |C| + \mbox{$\sum_{k=1}^m$} |A_k|,
\end{align}
where for example, $|C|$ denotes the matrix obtained by taking the absolute value of all elements of $C$. Furthermore,
\begin{align}\label{eq.13}
    M_{ij}^{\text{bound}} =  \begin{cases}  1 \quad \text{ if }  -\infty < L_{ij} \text{ or } U_{ij} <\infty
    \\    0 \quad \text{ otherwise. }\end{cases}
\end{align}
The aggregate pattern matrix $M^{\text{agg}}$ may be defined as $M^{\text{par}}+M^{\text{bound}}$; however, a more useful representation is a graph $\cG(\cV,\cE)$ where $\cV = \{1, 2, \ldots, n \}$ and the set of edges $\cE$ encodes $M^{\text{agg}}$. In other words,
\begin{align}\label{eq.14}
    \cE = \{ (i,j) \mid M^{\text{par}}_{ij} \not= 0, \; i<j \}\cup \{ (i,j) : M^{\text{bound}}_{ij} \not = 0 , \; i <j\}.
\end{align}
In the literature, $\cE$ is also referred to as the set of structurally non-zero entries across all data matrices and bounds. In practice, the aggregate pattern is the set of index pairs $(i,j)$ whose corresponding matrix entry $X_{ij}$ is relevant in the computation of the objective function as well as the affine/bound constraints in the original SDP.

\subsection{Clique decomposition of \texorpdfstring{$\mathbb{S}^n_+$}{S n plus}} \label{subsec-clique-conversion} Having established an SDP's aggregate sparsity pattern $\cE$ and its corresponding graph representation $\cG(\cV,\cE)$, conversion into SDP with \textit{quasi-block diagonal} structure can be made using some appropriate clique decomposition \cite{BP1993} and reordering. If $\cG(\cV,\cE)$ is chordal\footnote{The graph $\cG(\cV,\cE)$ is said to be chordal if every cycle of length at least 4 has a chord, which is an edge between two non-consecutive vertices in a cycle.}, one may choose a family of \textit{maximal cliques} (i.e., cliques that are not subgraphs of any other cliques) $\cC_1, \cC_2, \ldots, \cC_p \subseteq \cV$ for some $p \leq n$ such that $\cV= \cup_{l=1}^p \cC_l$ and $\cE \subseteq \cup_{l=1}^p \cC_l\times \cC_l$ and the resulting converted SDP problem is equivalent to the original SDP problem \eqref{eq.1}.

{When used as a matrix sparsity pattern, $\cE$ includes all diagonal positions and both $(i,j)$ and $(j,i)$ for each edge.} To facilitate further discussions, it is useful to make use of the following notation introduced in \cite{NFF2003}:
\begin{align}
    & \S^n(\cE,0)  =\{  X \in \S^n  :  X_{ij} =0 \hspace{0.2cm}  \text{for } (i,j) \notin \cE, \; i\neq j \},\label{eq.15}  \\[7pt]
    &\S^n_+(\cE,0)  = \mathbb{S}^n(\cE,0) \cap \mathbb{S}^n_+, \label{eq.16} \\[7pt]
    & \S^n(\cE,?) =  \mbox{the set of $n \times n$ symmetric matrices with elements specified in $\cE$},
    \\[7pt]
    &\S^n_+(\cE,?) = \{ X \in \mathbb{S}^n(\cE,?) :
    \exists\; M \succeq 0\; \mbox{such that}\; M_{ij} = X_{ij} \forall \; (i,j)\in\cE\}.\label{eq.17}
\end{align}
Here, $ \S^n(\cE,0)$ is the set of $\cE\text{-patterned $\S^n$ matrices}$, $\S^n_+(\cE,0)$ is the set of $\cE\text{-patterned PSD matrices}$ and $\S^n_+(\cE,?)$ is the set of PSD-completable matrices. Note that an element $X \in \mathbb{S}^n_+(\cE,?)$ has \textit{specified} entries in $\cE$ but may otherwise admit any sparsity pattern, though it is necessary that a PSD matrix must be recovered by possibly modifying (the unspecified) entries of $X$ in the complement $\cE^c$.

Assume that $\cG(\cV,\cE)$ is a chordal graph with a selected family of maximal cliques $\cC_1, \cC_2, \ldots, \cC_p$. Since these cliques cover $\cE$, one may relegate an analysis of the entire $\cG(\cV,\cE)$ to independent analyses of each clique. This idea was formalized in the work of \cite{Grone} where a sufficient condition for $X \in \mathbb{S}^n_+(\cE,?)$ is formulated in terms of submatrices of $X$.

\begin{theorem}{(Grone, 1984 \cite{Grone}):}
    \label{thm-1}
    $X \in \mathbb{S}^n_+(\cE,?) $ if and only if
    \begin{align*}
        X^{(l)}  := E_{\mathcal{C}_l} X E^\top_{\mathcal{C}_l} \in \mathbb{S}^{|\mathcal{C}_l|}_+, \quad \forall\; l=1,\ldots,p,
    \end{align*}
    where $E_{\mathcal{C}_l}$ is a $0$--$1$ matrix in $\R^{t \times n}$ with $t=|\cC_l|$, defined as follows: let the elements of $\cC_l$ be denoted as $\cC_l(1) < \ldots < \cC_l(t)$, then
    \begin{align}\label{eq.18}
        (E_{\cC_l})_{ij} = \begin{cases} 1 \hspace{0.2cm} \text{ if } j = \cC_l(i) \\ 0 \hspace{0.2cm} \text{ otherwise,}\end{cases} \text{for all } i = 1, \ldots t \text{ and } j = 1, \ldots, n.
    \end{align}
\end{theorem}

From the definition, the operation $E_{\mathcal{C}_l} X E^\top_{\mathcal{C}_l}$ extracts a submatrix $X^{(l)}$ of $X$ by extracting the entries of the rows and columns specified by $\mathcal{C}_l$.
Theorem 1 implies that in solving an SDP with chordal sparse aggregate matrix, the constraint $X \in \S^n_+$ is attainable by imposing the constraint  $X^{(l)} \in \S^{|\cC_l|}_+$ in the smaller cones $\S^{|\cC_l|}_+$ for $l=1,\ldots,p$. More precisely, positive-semidefiniteness of all the $X^{(l)}$ variables guarantees that $X \in\mathbb{S}^n_+(\cE,?)$. While this does not mean that the optimal solution of the original SDP can be directly recovered in the converted SDP, it is guaranteed that the original optimal value will be obtained by the latter; this is true because entries on $\cE^c$ have no relevance in both the objective and the constraints.

One may extend the benefits of the above results for non-chordal patterned SDPs by means of chordal completion/fill-ins. In practice, the sparsity information of an SDP is first extracted in the manner of \eqref{eq.12}--\eqref{eq.14}. This is followed by node reordering using some heuristics -- we typically use the approximate minimum degree (amd) ordering -- to retain maximum sparsity upon the subsequent chordal completion. The latter can be directly executed, for example, by running {\sc Matlab}'s {\tt chol} function on {a symmetric positive definite matrix with the same sparsity pattern as} the reordered $M^{\text{agg}}$.

Since the above chordal completion is always applied, in the rest of this paper we shall assume that all SDP instances are imbued with a chordal pattern $\cE$ and one can readily choose a corresponding family of covering maximal cliques $\cC_1, \cC_2, \ldots, \cC_p$. Furthermore, we refer to quantities such as $X^{(l)}$ as \textit{clique variables} due to their definition in Theorem 1; note that each clique variable corresponds to one sub-block of $X$ as illustrated in Figure~\ref{fig:quasi-block}. In other words, the choice of maximal cliques determines the \textit{quasi-block diagonal} structure and the computation-relevant portion of $X$ is exactly given by the superposition of all $X^{(l)}$.

The selection of maximal cliques is thus crucial in designing efficient conversion schemes. At a minimum, their individual size distribution must satisfy that the total cost of imposing the constraints $X^{(l)} \in \mathbb{S}^{|\mathcal{C}_l|}_+$, $l = 1, \ldots , p$, is significantly cheaper than the cost of imposing $X \in \S^n_+$. We shall discuss our proposed strategy of clique decomposition in greater detail in Section~\ref{subsec-3.2}.

To end this section, we consider a concrete attempt of SDP decomposition to illustrate another important aspect in the conversion process: the overlap structure of the clique variables. Consider a direct application of Theorem 1 on a single-block SDP \eqref{eq.1}, where $X \in \S^n_+$ is replaced with $X \in \S^n_+(\cE,?)$ i.e., $X^{(l)} \in \mathbb{S}^{|\mathcal{C}_l|}_+$ for all $l = 1, \ldots, p$. For the rest of this paper, we shall write $ \S^{|\mathcal{C}_l|}_+$ as $\S^{\mathcal{C}_l}_+$ for notational simplicity and also define
\begin{align}\label{eq.19}
\N^{\cC_l}    & =  \{ Y^{(l)}\in \S^{\cC_l} : L_{ij} \leq Y^{(l)}_{ij} \leq U_{ij}  \hspace{0.2cm}  \text{for }  (i,j)\in \cC_l\times \cC_l\}.
\end{align}
Note that unlike the $\S^n_+$ constraint, entrywise bound constraints are always separable. With these, we may now attempt a direct SDP decomposition as follows:
\begin{eqnarray*}
\begin{array}{rl}
\min &  \sum_{l=1}^p  \langle C^{(l)},X^{(l)} \rangle \\[5pt]
\text{s.t.} &   \sum_{l=1}^p  \langle A^{(l)}_i, X^{(l)}\rangle \;=\;b_i,\quad i=1,\ldots,m
\\[5pt]
& ( X^{(1)}, \ldots,  X^{(p)})  \in  \S^{\mathcal{C}_1}_+ \times \cdots \times \S^{\mathcal{C}_p}_+, \quad
( X^{(1)}, \ldots,  X^{(p)}) \in \N^{\mathcal{C}_1} \times \cdots \times \N^{\mathcal{C}_p},
\end{array}
\end{eqnarray*}
where the last constraint comes from $X \in \N^n$. However, the above generally does not recover a solution in $\S^n_+(\cE,?)$ since each $X^{(l)}$ is (inaccurately) treated as independent variables. The optimal solutions of the above decomposed problem will coincide with those of \eqref{eq.1} only in the special case where $\cC_{lr} := \cC_l\cap \cC_r =\emptyset$ for all $l\not=r$. For a general \textit{quasi-block diagonal} pattern, some additional constraints must be included to ensure that entries in any nonempty intersection $\cC_l \cap \cC_r$ are consistent, i.e., admit the same values in $X^{(l)}$ and $X^{(r)}$ for all entries corresponding to the nonempty intersection $\cC_l \cap \cC_r$. Details on the constructions of these consistency constraints are described in Section~\ref{subsec-2.2}. We remark that such a construction is called the conversion method in \cite{FKM2000} and the d-space conversion method using clique trees in
\cite{KKMY2011}.
By imposing the consistency constraints {for $(\cC_l,\cC_r)\in\cF$}, we get a multi-block SDP of the form:
\begin{eqnarray}
\begin{array}{rl}\label{eq.20}
\min &  \sum_{l=1}^p  \langle C^{(l)},X^{(l)} \rangle \\[5pt]
\text{s.t.} & \sum_{l=1}^p  \langle A^{(l)}_i, X^{(l)}\rangle \;=\; b_i, \quad i=1,\ldots,m
\\[5pt]
& \cL(X^{(1)},\ldots,X^{(p)}) \;=\; 0,
\\[5pt]
& ( X^{(1)}, \ldots,  X^{(p)})  \in  \S^{\mathcal{C}_1}_+ \times \cdots \times \S^{\mathcal{C}_p}_+, \quad
 ( X^{(1)}, \ldots,  X^{(p)}) \in \N^{\mathcal{C}_1} \times \cdots \times \N^{\mathcal{C}_p},
\end{array}
\end{eqnarray}
where $\cL$ {imposes the consistency constraints in \eqref{eq.25}}.

\subsection{Clique tree properties and consistency constraints} \label{subsec-2.2} A straightforward approach to ensure consistency between elements across all clique variables involves manually adding linear equalities to link the overlapping entries. Suppose that $(i,j) \in (\cC_l\times \cC_l) \cap (\cC_r\times \cC_r)$, i.e., $(i,j)$ is an entry that falls within the overlap $\cC_l \cap \cC_r$. One can directly impose $X^{(l)}_{[ij]} - X^{(r)}_{[ij]} = 0$. Here we use the notation $[ij]$ to denote the corresponding index pairs in the smaller matrices $X^{(l)}$ and $X^{(r)}$. Of course, since any such constraints are inter-clique in nature, this approach precludes the possibility of completely decomposing the problem.

Furthermore, the above set of consistency constraints is not uniquely determined from $\cC_1, \cC_2, \ldots, \cC_p$. The number of such constraints may in fact vary depending on the clique ordering. In this work, we closely follow the maximal-clique ordering strategy proposed by \cite{NFF2003} to generate the set of constraints in a consistent and predictable manner. First, we define a \textit{clique tree} $\cT(\cK,\cF)$ with $\cK = \{\cC_1, \cC_2, \ldots, \cC_p \}$ and $\cF$ that is constructed to adhere to the following properties:

\begin{definition}{}
(a) {The clique intersection property (CIP)}
				\begin{align}\label{eq.22}
				\forall \cC_l, \cC_r \in \cK, \hspace{0.1cm} \cC_l \cap \cC_r \subseteq \cC_s \text{ for every clique } \cC_s \text{ on the unique path between } \cC_l \text{ and } \cC_r .
				\end{align}
(b) {The running intersection property (RIP)}
				\begin{align}\label{eq.23}
				\forall l=1,\ldots,p-1,\; \exists r \in \{l+1,\ldots,p\}: \hspace{0.2cm} \cC_l \cap ( \cC_{l+1} \cup \ldots \cup \cC_p) \subsetneq \cC_r.
				\end{align}
\end{definition}

The purpose of CIP is twofold: it is used as a framework for the tree construction itself and it imposes a structure that allows convenient enumeration of overlapped entries. Note that the \textit{separator set} $\{\cC_l \cap \cC_r \mid (\cC_l,\cC_r) \in \cF \}$ is invariant to the choice of $\cT(\cK,\cF)$ as long as \eqref{eq.22} holds. Given one such tree, we may deduce the number of pairwise overlaps by considering only intersections of adjacent cliques, i.e., the elements of $\cF$. The number of additional constraints is given by
\begin{align}\label{eq.24}
m^+ = \sum_{(\cC_l,\cC_r) \in \cF} |\cC_l \cap \cC_r|(|\cC_l \cap \cC_r|+1)/2.
\end{align}

On the other hand, the RIP imposes a particular numbering of the cliques that respects the \textit{perfect elimination ordering} -- an ordering guaranteed to exist in a chordal graph. With condition \eqref{eq.23}, the root node will be labeled as $p$ and the numberings along any path traced from this node to any leaf node will be a decreasing sequence. Hence in practice, a concrete algorithm to determine the set of overlaps -- and thus, the additional constraints -- can be constructed; it simply involves identifying paths from the root to each leaf node.

For more information on clique trees, we refer the reader to \cite{BP1993}. For the purpose of our discussion, the above conditions are sufficient to allow a complete description of the additional constraints in a concise and cost-effective manner:
\begin{align}\label{eq.25}
X^{(l)}_{[ij]} - X^{(r)}_{[ij]} = 0   \hspace{0.2cm} \text{for } (i,j) \in (\cC_l\times \cC_l) \cap (\cC_r\times \cC_r), \;  i\leq j,\; {(\cC_l,\cC_r)\in\cF},
\end{align}
where the notation $[ij]$ refers to the corresponding index pairs in the smaller matrices $X^{(l)}$ and $X^{(r)}.$ {By the CIP, the same entry has the same value in all clique variables containing it.}

\section{Clique decomposition strategies} \label{sec-clique-decomp} Having established the theoretical aspects of maximal cliques selection and their properties, in this section we elaborate some practical strategies applicable to real applications. It is important to note that all strategies described in this section adhere to the CIP and RIP so that all previously described properties remain valid.

\subsection{Chordal decomposition for arrowhead-patterned SDPs} \label{sec-maxnorm} We begin by analyzing SDPs from the class of problems that naturally admit an \textit{arrowhead} sparsity pattern. In our work, we use this term to describe the general shape of a sparsity pattern under {\tt amd} reordering and chordal completion -- a sufficiently sparse SDP typically exhibits a structure resembling a (series of) ``arrowhead'' after this process. The max-norm matrix recovery SDP is a special case where this pattern is observed under the default ordering. We shall therefore consider these problems for a clearer illustration of the process of selecting clique distributions.

For the max-norm matrix recovery problem \cite{FLT2018}, our objective is to recover a low-rank matrix $M^*\in \R^{d_1\times d_2}$ (we assume $d_1\geq d_2$ without loss of generality) from an incomplete sample matrix $A\in \R^{d_1\times d_2}$. In other words, $A$ is obtained from $M^*$ by randomly removing some entries from the latter. The distribution of the retained entries is non-uniform and their values are possibly noise contaminated. In the noiseless case, $A_\cJ = M^*_\cJ$ and $A_{ij} =0$ for all $(i,j) \in \cJ^c$, where $\cJ$ refers to the sampled index pairs. Let $\widehat{A} = [0_{d_1\times d_1},\, A;\, A^\top,\, 0_{d_2\times d_2}]\in \S^{d_1+d_2}$ and define the corresponding lifted observation set
\[
{\widehat{\cJ}:=\{(i,d_1+j):(i,j)\in\cJ\}}.
\]
An approximate recovery can be obtained by solving the following SDP:
\begin{align}\label{eq.26}
 \min_{X , \theta } & \; \sum_{(i,j)\in {\widehat{\cJ}}} |X_{ij}-\widehat{A}_{ij}|+\lambda \theta \nonumber \\[5pt]
\mbox{s.t.} & \;\norm{M}_\infty \leq \alpha, \quad \diag{X}  \leq \theta e \nonumber\\[5pt]
  & \; X = \left[ \begin{array}{cc} W_1 &  M \\ M^\top &  W_2\end{array}\right] \in \S^{d_1+d_2}_+,  \hspace{0.2cm} \theta \in \R,
\end{align}
where $e$ is a vector of all-ones, $\alpha > 0$ is a known uniform upper bound of $M^*$ and $\lambda > 0$ is a regularization parameter that penalizes $\theta$, the common upper bound on the diagonal entries of the lifted PSD matrix $X$, as a convex surrogate for $\operatorname{rank}(M)$. Meanwhile, $\norm{M}_\infty=\max_{ij}|M_{ij}|\leq\alpha$ is a separate entrywise bound. The above SDP formulation is due to \cite{SRJ2004}.

Note that a different SDP formulation using the trace norm $\norm{M}_*$ as a surrogate function of $\text{rank}(M)$ is also common in matrix recovery. Such a formulation possesses a similar sparsity structure, but is considered unsuitable for our applications with non-uniform sampled $\cJ$ -- see for example, \cite{CZ2013} and \cite{SS2005} for further discussions.

As a concrete illustration, consider a simple case with $d_1 = 20$ and $d_2 = 10$ and the following family of two maximal cliques covering the aggregate sparsity pattern of $X$:
\begin{align*}
\cC_1    =\{1,\ldots, 10\}\cup\{21,&\ldots,30\}, \quad \cC_2 =\{11,\ldots,30\}, \; \text{ and }  
\cC_1\cap \cC_2 = \{ 21,\ldots,30 \}.
\end{align*}
With such a configuration, the structure of problem \eqref{eq.26} is illustrated in Figure~\ref{fig:1a}. The images depict the original aggregate sparsity pattern of \eqref{eq.26}, its arrowhead sparsity pattern is observed upon chordal completion without reordering, and the entries corresponding to the cliques $\cC_1$ and $\cC_2$.

\begin{figure}[ht]
    \begin{center}
        \includegraphics[clip=true,width=8.5cm,height=4.4cm]{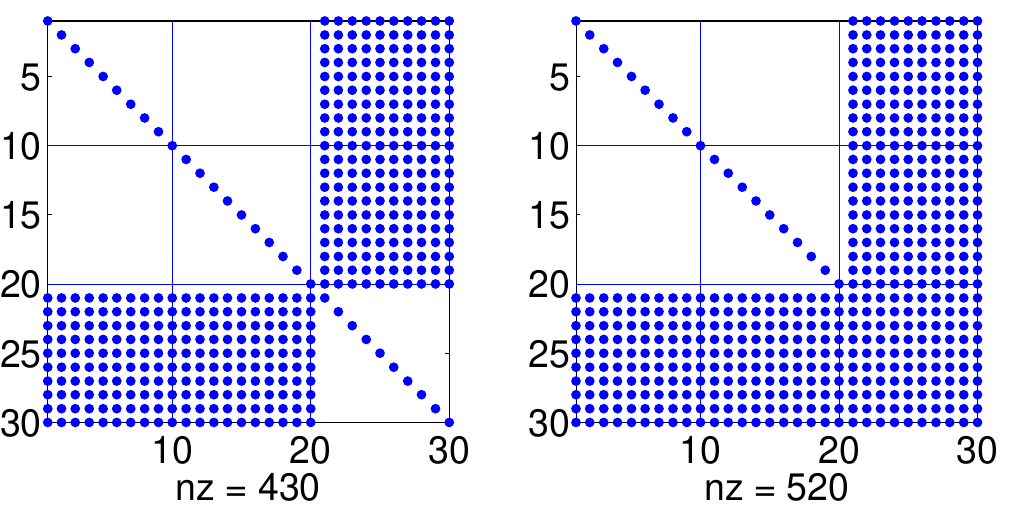}
        \\[5pt]
        \includegraphics[clip=true,width=8.5cm,height=4.4cm]{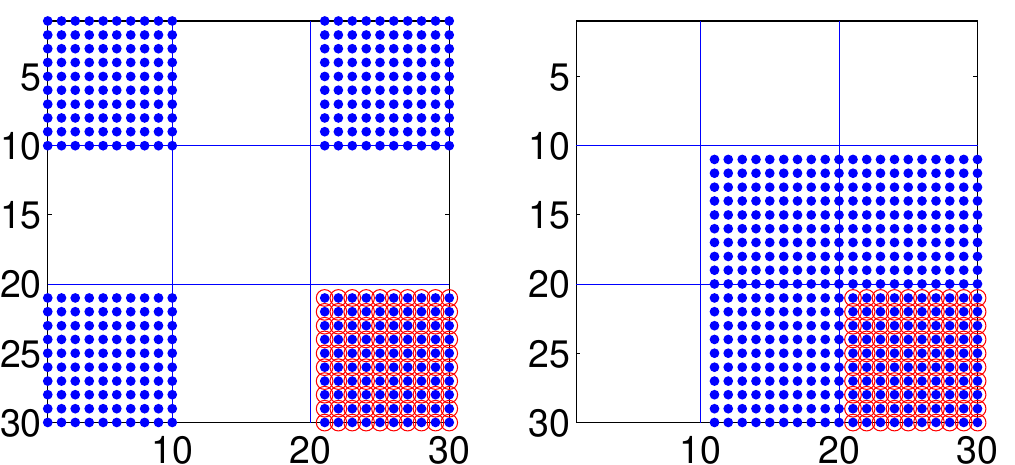}
        \caption{(Upper panel) aggregate sparsity pattern of the SDP \eqref{eq.26} and its chordal completion. (Lower panel) Entries covered by $X^{(1)}$ (left) and $X^{(2)}$ (right), where $\cC_1 = \{1,\ldots,10\}\cup\{21,\ldots,30\}$ and $\cC_2=\{ 11,\ldots,30\}$. The entries with red circles denote the intersection $(\cC_1\cap \cC_2) \times (\cC_1\cap \cC_2)$.}
        \label{fig:1a}
    \end{center}
\end{figure}

Converting \eqref{eq.26} to the corresponding two-block SDP yields:
\begin{align}\label{eq.27}
    \min_{X^{(1)} ,X^{(2)} , \theta } &\;\sum_{l=1}^2 \sum_{(i,j)\in \widehat{\cJ}\cap (\cC_l\times \cC_l)} | (E^\top_{\cC_l} X^{(l)} E_{\cC_l})_{ij} -\widehat{A}_{ij}|+ \lambda \theta \nonumber \\[8pt]
    \mbox{s.t.} &  \;  \diag{X^{(1)}} \leq \theta e, \hspace{1cm} \diag{X^{(2)}} \leq \theta e \nonumber  \\[5pt]
    & \; \norm{X^{(1)}_{1:10,11:20}}_\infty \leq \alpha, \quad  \;  \norm{X^{(2)}_{1:10,11:20}}_\infty \leq \alpha \nonumber \\[5pt]
    & \; X^{(1)} \in \S^{20}_+, \hspace{0.6cm} \quad X^{(2)} \in \S^{20}_+, \hspace{0.6cm} \theta \in \R.
\end{align}
To completely recover the optimal solution of \eqref{eq.26}, solving \eqref{eq.27} together with the consistency constraints $X^{(1)}_{[ij]} - X^{(2)}_{[ij]} = 0  \hspace{0.2cm} \text{for }  (i,j)\in (\cC_1\cap \cC_2)\times (\cC_1\cap \cC_2), \; i\leq j$ is once again necessary.

In converting \eqref{eq.26} to \eqref{eq.27}, we have replaced the single PSD constraint on $X \in \S^{30}_+$ with a pair of PSD constraints:
$X^{(1)}\in \S^{20}_+$ and $X^{(2)}\in \S^{20}_+$. Note that the expected complexity for a spectral decomposition on $X$ is of the order $(30)^3 = 27,000$ whereas for the two $\S^{20}_+$ variables, it is only $2 (20)^3 =16,000$: a theoretical reduction ratio of $16/27$. This ratio is in fact applicable to all problems with $d_1 = 2 d_2$ since they admit the same form of converted SDP as \eqref{eq.27}.

Based on our empirical observation, theoretical ratios such as the above are largely reliable in predicting the efficiency of the SDP conversion. In larger problems where $d_1= K d_2$ for some integer $K\geq 2$, the ratio can be further improved depending on the cliques configuration. In the following, we design a ratio minimizing decomposition strategy for general recovery problems.

A natural attempt to decompose an $\S^{d_1+d_2}$ variable with $d_1 = Kd_2$ is to uniformly partition its first $d_1 \times d_1$ block into $K \times K$ blocks of size $d_2\times d_2$. The clique variables are formed by pairing each of the latter with the bottom-right $d_2 \times d_2$ block. Explicitly, the cliques are given by
\begin{align*}
    \cC_l   &=\{(l-1)d_2+1,\ldots, l d_2\}\cup\{K d_2+1,\ldots, (K+1)d_2\}, \quad  \text{ and } \\[5pt]
    \cC_l \cap \cC_r &= \{ K d_2 +1,\ldots, (K+1)d_2 \}  \hspace{0.2cm} \text{for } \text{all distinct } l,r\in\{1,\ldots,K\}.
\end{align*}
Here, \eqref{eq.26} is converted to an SDP with $K$ clique variables each in $\S^{2d_2}_+$; this corresponds to an expected spectral decomposition complexity of order $K (2d_2)^3$ whereas the original corresponds to $(d_1+d_2)^3 = ((K+1)d_2)^3$, yielding a reduction ratio of $(8K)/(K+1)^3$. Hence, the benefit of conversion is enhanced as $K$ increases. Additionally, the memory required to store the variables -- which scales quadratically with size -- is expected to be reduced by a factor of $(4K)/(K+1)^2$.

Despite the promising reduction, the above strategy can still be improved by considering finer partitions to generate the clique variables. Following the above partition of the first $d_1 \times d_1$ block into $d_2\times d_2$ sub-blocks, we further partition each of the latter into blocks of size $(d_2/L) \times (d_2/L)$. The cliques are again obtained by pairing these fine blocks with the bottom-right $d_2 \times d_2$ block, resulting in $KL$ variables of size $(d_2/L+d_2)\times (d_2/L+d_2)$. The corresponding reduction ratio is $\frac{KL(1/L+1)^3}{(K+1)^3}$, whence it is not difficult to deduce that selecting $L=2$ yields the optimal ratio.

To conclude, {letting $K:=d_1/d_2$}, we adopt the following strategy for the class of matrix completion problems \eqref{eq.26}, assuming that {$d_2$ is even and $2K=2d_1/d_2$ is an integer}:
\begin{align*}
    \cC_l   &=\{(l-1)d_2/2 + 1,\ldots, l d_2/2\}\cup\{K d_2+1,\ldots, (K+1)d_2\}, \quad  \text{ and } \\[5pt]
    \cC_l \cap \cC_r &= \{ K d_2 +1,\ldots, (K+1)d_2 \}  \hspace{0.2cm} \text{for } \text{all distinct } l,r\in\{1,\ldots,2K\},
\end{align*}
and the corresponding converted SDP is as follows:
\begin{align}\label{eq.28}
    \min_{X^{(1)}, \ldots,  X^{(2K)},\; \theta } &\;\sum_{l=1}^{2K} \sum_{(i,j)\in {\widehat{\cJ}}\cap (\cC_l\times \cC_l)} |(E^\top_{\cC_l} X^{(l)} E_{\cC_l})_{ij}-\widehat{A}_{ij}|+ \lambda \theta  \\[8pt]
    \mbox{s.t.} &  \;  \diag{X^{(l)}} \leq \theta e  \hspace{0.2cm} \text{for }    l = 1, \ldots, 2K  \nonumber \\[5pt]
    & \; \norm{X^{(l)}_{1:d_2 /2,1+d_2/2:3d_2/2}}_\infty \leq \alpha   \hspace{0.2cm} \text{for }  l = 1, \ldots, 2K \nonumber \\[5pt]
    & \; ( X^{(1)}, \ldots,  X^{(2K)})  \in  \S^{3d_2/2}_+ \times \cdots \times \S^{3d_2/2}_+ \nonumber \\[5pt]
    & \; X^{(l)}_{[ij]} - X^{(1)}_{[ij]}  = 0,   \hspace{0.2cm} \text{for all }   (i,j)\in (\cC_l\cap\cC_1)\times(\cC_l\cap\cC_1) \text{ and } i\leq j,   \text{ for each }  l = 2, \ldots, 2K. \nonumber
\end{align}

\subsection{CMD: {clique decomposition based on PSD projection cost}} \label{subsec-3.2} We now consider SDPs with a generic sparsity pattern, where reordering of nodes (row and column permutations) is vital to minimize fill-ins during the chordal completion and to find good clique distributions. Once again, we remark that arrowhead-like \textit{quasi-block diagonal} structures such as those observed in Figure~\ref{fig:1a}(c) and (d) are typically observed after reordering for general SDPs with reasonably conducive sparsity pattern. Thus, clique construction generally also involves the grouping or decomposition of blocks akin to the procedure implemented in the previous section.

Complications typically arise due to the non-uniformity of the block sizes, where it is no longer possible to generate uniform cliques as in the matrix recovery problems. In this section, we propose a greedy heuristic to determine suitable clique assignments. {We use two criteria to assess each clique assignment, namely the estimated PSD projection cost} and the number of additional equality constraints (the linking constraints described in Section~\ref{subsec-2.2}), which is equal to $$m^+=  \sum_{(\cC_l,\cC_r)\in\cF}  |\cC_l \cap \cC_r|(|\cC_l \cap \cC_r|+1)/2.$$

The following result from \cite{NFF2003} shows how to merge {two cliques connected by an edge of the clique tree} while preserving the clique tree and the CIP in \eqref{eq.22}.

\begin{theorem}{(Nakata et al., 2003)}
    Consider a chordal graph $\cG(\cV,\cE)$ with a selection of maximal cliques $\cC_1, \ldots, \cC_p$ corresponding to a clique tree $\cT(\cK,\cF)$. Let $(\cC_l, \cC_r) \in \cF$, and suppose that $\cC_l$ and $\cC_r$ are merged into one clique $\cC_l \cup \cC_r$. This defines a graph $\cT'(\cK',\cF')$ with:
    \begin{align*}
        \cK' &=   \Big(\cK \setminus \{\cC_l, \cC_r\} \Big) \cup \{  \cC_l \cup \cC_r \} \hspace{0.2cm} \text{ and }\\[7pt]
        {\cF}' &= {\big\{(\cC,\cC')\in\cF:\cC,\cC'\notin\{\cC_l,\cC_r\}\big\}} \\[5pt]
        &\quad {\bigcup\big\{(\cC_l\cup\cC_r,\cC):\cC\in\cK\setminus\{\cC_l,\cC_r\},\;
        (\cC_l,\cC)\in\cF\ \mbox{or}\ (\cC_r,\cC)\in\cF\big\}}.
    \end{align*}
    Then $\cT'(\cK',\cF')$ is a clique tree that respects CIP.
\end{theorem}

In the above theorem, $\cT'(\cK',\cF')$ is a clique tree of size $p-1$ whose nodes and edges are those of $\cT(\cK,\cF)$ except that the nodes $\cC_l$ and $\cC_r$ are fused into a single node. Note that after any number of merges, the final clique tree can also be numbered to satisfy the RIP.

Algorithm~\ref{algo-CMD} presents the pseudocode for our clique-assignment strategy named CMD. This algorithm {starts from a clique tree and merges only pairs connected by an edge of the current tree}. {The parameter $\delta=g(\max\{|\cC_l|,|\cC_r|\})$ sets the threshold for the cost ratio in Condition 1. The parameters $\tau_2$ and $\tau_3$ set the thresholds for the overlap ratio and clique size in Conditions 3 and 4, respectively. The choice of $g$ and the default parameter values are given in Section~\ref{sec-experimental-setting}.} {Condition 1 aims to reduce the estimated PSD projection cost.} Algorithm~\ref{algo-CMD} also promotes merging when $(\cC_l,\cC_r)$ is a pair of small cliques or when they share a significant portion of common elements. The latter is partially inspired by SparseCoLO \cite{NFF2003}, where the main priority is to minimize the number of additional consistency constraints corresponding to the overlapping common elements of any two clique variables.

Note that the initial set of cliques can be directly deduced from the pattern observed after the (symbolic) Cholesky factorization -- for instance, we may group non-zero entries along the same row into a single clique. Of course, cliques obtained in this manner rarely correspond to a good distribution: they are typically highly unbalanced in clique sizes and require numerous merges.

\IncMargin{1em}
\begin{algorithm}[ht!]
\caption{CMD: {clique decomposition based on PSD projection cost}}
\label{algo-CMD}
\SetAlgoLined
\textbf{Initialize:} {A clique tree $\cT(\cK,\cF)$ with $\cK=\{\cC_1, \dots, \cC_p\}$}, and parameters $\tau_2 \in [0,1]$ and $\tau_3$. \\
{Set ${\tt merged}\leftarrow {\tt true}$.}\\

\While{{${\tt merged}={\tt true}$}}{
	{Set ${\tt merged}\leftarrow {\tt false}$.}\;
		\ForEach{{$(\cC_l,\cC_r)\in\cF$}}  {
			\BlankLine
			Compute  $\delta \; \leftarrow \; {g(\max\{|\cC_l|,|\cC_r|\})}$ \;
			\BlankLine
			$R_{lr} \; \leftarrow \; \frac{|\cC_l  \cup \cC_r|^3}{|\cC_l|^3 + |\cC_r|^3}$\;
			\BlankLine
			Condition 1: $ 1+ \delta > R_{lr}$
			\BlankLine
			Condition 2: $ |\cC_l \cap \cC_r| = \min \{ |\cC_l | , |\cC_r| \}$
			\BlankLine
			Condition 3: $ |\cC_l \cap \cC_r| > \tau_2 * \max \{ |\cC_l | , |\cC_r| \}$
			\BlankLine
			Condition 4: $ \max \{ |\cC_l | , |\cC_r| \} <  \tau_3$
			\BlankLine
			\If{ Any Condition $i$ is satisfied, $i = 1, \ldots, 4$ } {
			\BlankLine
					Merge $\cC_l$ with $\cC_r$\;
					{Update $\cT(\cK,\cF)$ as in Theorem 2 and set $p\leftarrow|\cK|$}\;
					{Set ${\tt merged}\leftarrow {\tt true}$ and restart the scan of $\cF$}\;
					\Break\;
				}
			\BlankLine
	}
}
\end{algorithm}
\DecMargin{1em}

The {clique distribution} obtained from CMD fundamentally differs from those obtained from SparseCoLO since we {use the estimated PSD projection cost ratio $\sum_{l=1}^p |\cC_l|^3/n^3$ as the main criterion} instead of {the number of overlapping entries}. Our rationale shall be made clear in Section~\ref{sec-computational}, where our efficient approach in imposing the consistency constraints allows us to handle a much higher number of overlaps.

The efficacy of the various decomposition schemes proposed in Section \ref{sec-clique-decomp} is highly contingent on the original sparsity pattern of the problem. For problems with a relatively dense aggregate pattern, Algorithm~\ref{algo-CMD} typically outputs a trivial set containing a single clique that represents the entire variable. Note that for some SDP problems, it is possible to perform a pre-processing step to obtain a reformulated problem that has a conducive aggregate sparsity pattern. In Appendix~\ref{sec-theta}, we present a technique for sparsifying a class of seemingly dense SDP with a sparse-plus-low-rank cost matrix. For more complex problems where the constraint matrices are also sparse-plus-low-rank, we refer to \cite{Tang2026-SPLR} for more sophisticated sparsification techniques.

\section{Computational details on multi-clique structures} \label{sec-computational} In the practical implementation, our naive approach in Section~\ref{subsec-clique-conversion} \eqref{eq.20} requires modifications on two fronts: (i) the matrix coefficients on the objective and constraint functions need to be adjusted to avoid overcounting in the overlapped indices and (ii) the set of linear constraints \eqref{eq.25} must be imposed either directly, or using a cheaper approach. We remark that all computational details presented in this section are applicable to any form of decomposition, even when $\{\cC_1, \ldots, \cC_p\}$ is not a family of maximal cliques from a chordal graph.

The following quantities are useful in facilitating our subsequent discussions:
\begin{align}
    \widehat{E} & =  \cup_{l=1}^p \cC_l\times \cC_l  \label{eq.34} \\[5pt]
    \cO_{ij}      & = \{l : (i,j)\in  \cC_l\times \cC_l,\;\; l=1,\ldots,p \} \hspace{0.2cm} \text{for } (i,j) \in \widehat{E} \label{eq.35}\\[5pt]
    l_{ij}       & =\min \cO_{ij}=\min \{ l : (i,j)\in \cC_l\times \cC_l, \; \; l=1,\ldots,p\} \hspace{0.2cm} \text{for } (i,j)\in \widehat{E}. \label{eq.36}
\end{align}
Note that $\cO_{ij}$ is the set of cliques that contain the edge $(i,j)$ and each $|\cO_{ij}| \geq 1$; we say $(i,j)$ is an \textit{overlapped entry} if the latter is strictly greater than 1. Furthermore, $l_{ij}$ is the lowest clique index corresponding to $(i,j)$ and $\cC_{l_{ij}}$ can be considered its \textit{principal} clique.

\subsection{Implementing the decomposition and compensating the overlaps} \label{subsec-overlap-compensation} In \eqref{eq.20}, quantities such as $\sum_{l=1}^p\langle C^{(l)},X^{(l)} \rangle$ generally include multiple counts of identical terms. Suppose $|\cO_{ij}|>1$ for some entry $(i,j)$, then $|\cO_{ij}|$ terms identical in value as $\inprod{C_{ij}}{X_{ij}}$ are overcounted towards the sum total. To mitigate this issue, the following pre-processing steps are necessary prior to running any SDP solvers:
\begin{align}
    C_{ij} &\rightarrow \frac{C_{ij}}{|\cO_{ij}|} \quad \forall \; (i,j) \in \cC_l\times \cC_l, \; l=1,\ldots,p \hspace{0.2cm} \text{and similarly, }\label{eq.37}
    \\[7.5pt]
    (A_{k})_{ij} &\rightarrow \frac{(A_k)_{ij}}{|\cO_{ij}|} \quad \forall \; (i,j) \in \cC_l\times \cC_l, \; l=1,\ldots,p \hspace{0.2cm} \text{for all }k = 1, \ldots m. \label{eq.38}
\end{align}

In our practical implementation, the sets $\cO_{ij}$ are encoded as a matrix that can be constructed efficiently by first defining the \textit{decomposition matrix} $\cD^\top$ that satisfies the following identity:
\begin{align}\label{eq.39}
    \cD^\top \; \text{svec}(X) &= \text{svec}_F(X),
\end{align}
where $\text{svec}_F(X) :=\big[\text{svec}\left(X^{(1)}\right)\,;\, \ldots\,;\, \text{svec}\left(X^{(p)}\right)  \big]\in\R^{N_F}$ has dimension
\begin{align}\label{eq.40}
    N_F = \sum_{l=1}^p |\cC_l|(|\cC_l|+1)/2.
\end{align}
In the above, $\text{svec}(Y) = [Y_{11}; \sqrt{2}Y_{12}; Y_{22}; \ldots; \sqrt{2}Y_{1n};\sqrt{2}Y_{n-1,n}; Y_{nn}]$ for any $n\times n$ symmetric matrix $Y$.
Note that $\cD^\top\in \R^{N_F \times N}$ is a $0$-$1$ matrix with $N = |\text{svec}(X)| = n(n+1)/2$. On the other hand, $\cD \in \R^{N \times N_F}$ nearly gives the inverse operation of \eqref{eq.39} if not for the overlapping entries. Note that for an entry $(i,j)$, $\cD$ collects multiple terms identical to $X_{ij}$ from all $\cC_l$ that cover $(i,j)$. In other words, $\text{svec}(Z) \leftarrow \cD \; \text{svec}_F(X)$ is an operation that puts $\sum_{l \in \mathcal{O}_{ij}} X^{(l)}_{ij}$ on entry $(i,j)$ of $Z$.

From these properties, it is not difficult to deduce that $H:= (\cD \cD^\top)^\dagger \in \S^N$ is a diagonal matrix whose diagonal entries are the reciprocals of the overlap counts on $\widehat E$, namely, we have
\begin{align}\label{eq.41}
    [\operatorname{diag}(H)]_{(ij)}=
    \begin{cases}
        1/|\cO_{ij}|, & (i,j)\in\widehat E,\\
        0, & (i,j)\notin\widehat E,
    \end{cases}
    \quad 1\leq i\leq j\leq n.
\end{align}
Here, the notation $(ij)$ is used to denote the index of the entry $X_{ij}$ along the vector $\text{svec}(X)\in\R^N$. Note that the diagonal entries of $H$ are the overlap-compensation factors in \eqref{eq.37}. With $H$, modifications such as \eqref{eq.37} and \eqref{eq.38} can be cheaply executed even during the solving process. For example, we may transform
\begin{align}
    \text{svec}(C) & \rightarrow H \;\text{svec}(C).
    \label{eq.42}
\end{align}
When the computation aims to recover clique quantities, e.g., $X^{(l)}$, we simply extract entries corresponding to $\cC_l$ from $\text{svec}(X)$. That is, construct $\cD^\top_l \in \{0,1\}^{N_l \times N}$ so that $\cD_l^\top \; \text{svec}(X) = \text{svec}(X^{(l)})$ with $N_l = |\cC_l|(|\cC_l|+1)/2$. We can analogously define $\cM_l \in \R^{N_l\times N_F}$ as
\begin{align}
    \cM_l &= \cD^\top_l H \cD, \;\text{ so that }\cM_l \; \text{svec}_F(X) \text{ returns the averaged entries associated with clique }\cC_l,
    \label{eq.44}
    \\[5pt]
    \cM &= [\cM_1\,;\, \ldots\,;\, \cM_p].
\end{align}
With a slight abuse of notation, $\cM_l(V)$ denotes the clique matrix satisfying $\text{svec}(\cM_l(V))=\cM_l\,\text{svec}_F(V)$.

Finally, we note that all the above compensating (sparse) matrices can be cached during the pre-processing step such that any overlap compensating operations during the solution process amounts only to some sparse matrix-vector multiplications.

\subsection{Consistency constraints via linear-space projection} \label{sec-linproj} We are now in the position to introduce our strategy in handling the consistency constraints \eqref{eq.25}. First, we describe the feasible set corresponding to the above family of linear constraints as a simple linear space:
\begin{align}\label{eq.45}
    \L = \Big\{  (X^{(1)},\dots,X^{(p)}): X^{(l)}_{[ij]}=X^{(r)}_{[ij]} \hspace{0.2cm} \text{for }  (i,j)\in \cC_{lr}\times \cC_{lr}, \; i\leq j,\; (\cC_l,\cC_r)\in \cF
    \Big\}.
\end{align}
Recall that we use $[ij]$ to denote the index pair corresponding to $(i,j)$ in the sub-blocks $X^{(l)}$, $X^{(r)}$. In the McADMM algorithm described in Section~\ref{sec-ADMM}, imposing set constraints involves projections onto the corresponding sets. In the case of our multi-block SDP \eqref{eq.20}, together with the consistency constraints, these sets are $\S^{\mathcal{C}_1}_+ \times \cdots \times \S^{\mathcal{C}_p}_+$, $\N^{\mathcal{C}_1} \times \cdots \times \N^{\mathcal{C}_p}$, and $\L$.

In this section, we discuss an efficient approach concerning the projection onto the intersection of the last two sets; projections onto the PSD cones are more straightforward and shall be outlined in Section~\ref{sec-ADMM}. The following proposition claims that the projection of a matrix $X$ onto $\L\cap (\N^{\mathcal{C}_1} \times \cdots \times \N^{\mathcal{C}_p})$ can be done by a simple averaging step, followed by imposing the bound constraints in $\N^{\mathcal{C}_1} \times \cdots \times \N^{\mathcal{C}_p}$.

\begin{proposition} \label{prop-intersection}
    Given $V\in \S^{\cC_1}\times \cdots \times \S^{\cC_p}$, let $\widehat{V} := \Pi_{\L\cap (\N^{\cC_1}\times\cdots\times \N^{\cC_p})}(V)$. We have
    \begin{align}\label{eq.46}
        \widehat{V}^{(l)}_{[ij]} = \min\left\{U_{ij},\,\max\left\{ \frac{\sum_{t\in \cO_{ij}} V^{(t)}_{[ij]}}{|\cO_{ij}|}, L_{ij}\right\}\right\}
        \hspace{0.2cm} \text{for all } (i,j)\in \cC_l\times \cC_l, \; l=1,\ldots,p.
    \end{align}
    The computation of $\widehat{V}^{(l)}$ is implemented as
    \begin{align}\label{eq.49}
        \text{svec}(\widehat{V}^{(l)}) =\min \left\{  \text{svec}(U^{(l)}) , \max\left\{ \cM_l \; \text{svec}_F(V) ,  \text{svec}(L^{(l)})  \right\} \right\},\quad l=1,\ldots,p.
    \end{align}
\end{proposition}
\begin{proof}
    Note that $ \widehat{V}$ is given by the solution to the following minimization problem:
    \begin{align}\label{eq.47}
        \min_{X\in \L\cap (\N^{\cC_1}\times \cdots \times\N^{\cC_p})} \norm{X-V}^2.
    \end{align}
    For any $X\in \L$, we have that
    \begin{align*}
        X^{(l)}_{[ij]} = X^{(r)}_{[ij]} \hspace{0.2cm} \text{for all } r \in \cO_{ij} \text { and }(i,j)\in \cC_l\times \cC_l.
    \end{align*}
    From the definitions of $\widehat{E}$, $\cO_{ij}$ and $l_{ij}$ in \eqref{eq.34}-\eqref{eq.36}, we have
    \begin{align*}
        \norm{X-V}^2 =   \sum_{l=1}^p \sum_{(i,j)\in \cC_l\times \cC_l}  \left(X^{(l)}_{[ij]}-V^{(l)}_{[ij]}\right)^2 = \sum_{(i,j)\in\widehat{E}}\, \sum_{r\in \cO_{ij}} \left(X^{(l_{ij})}_{[ij]}-V^{(r)}_{[ij]}\right)^2.
    \end{align*}
    Thus, \eqref{eq.47} can be rewritten as
    \begin{align*}
        \min_{X\in \L\cap (\N^{\cC_1}\times \cdots \times\N^{\cC_p})} \Big\{  \sum_{(i,j)\in\widehat{E}}
        \sum_{r\in \cO_{ij}} \left(X^{(l_{ij})}_{[ij]}-V^{(r)}_{[ij]}\right)^2 :
        L_{ij} \leq X^{(l_{ij})}_{[ij]} \leq U_{ij}, \;\forall\; (i,j)\in\widehat{E}
        \Big\}.
    \end{align*}
    Solving the optimality conditions for this bounded convex problem yields the following:
    \begin{align*}
        \widehat{V}^{(l_{ij})}_{[ij]} = \min\left\{U_{ij},\, \max\left\{\frac{\sum_{r\in \cO_{ij}} V^{(r)}_{[ij]}}{|\cO_{ij}|},\; L_{ij}\right\}\right\}
        \hspace{0.2cm} \text{for } (i,j)\in\widehat{E}.
    \end{align*}
    Note that once again, $l_{ij}$ bears no importance other than being the representative clique index for the entry $(i,j)$. Thus, the above implies that \eqref{eq.46} holds for all cliques $\cC_l$.
\end{proof}

With the above proposition, define $\cB:=\L\cap(\N^{\mathcal{C}_1}\times\cdots\times\N^{\mathcal{C}_p})$. The proximal mapping of its support function $\delta_\cB^*$ can be evaluated via the Moreau decomposition:
\begin{align}\label{eq.50}
    \operatorname{Prox}_{\gamma\delta_{\cB}^*}(V)=V-\gamma\Pi_{\cB}(\gamma^{-1}V),\qquad \gamma>0.
\end{align}
We shall refer to the above framework as the \textit{linear-space projection} approach. Under this setup, a general multi-block SDP with bound constraints can be written as follows:
\begin{eqnarray}\label{eq.51}
    \begin{array}{rl}
        p^*&=  \min_{X^{(1)}, \ldots, X^{(p)}} \hspace{0.1cm}  \sum_{l=1}^p  \langle C^{(l)},X^{(l)} \rangle \\[5pt]
        \text{s.t.} & b_i =   \sum_{l=1}^p  \langle A^{(l)}_i, X^{(l)}\rangle \hspace{0.2cm} \text{for }  i=1,\ldots,m   \\[5pt]
        & ( X^{(1)}, \ldots,  X^{(p)})  \in  \S^{\mathcal{C}_1}_+ \times \cdots \times \S^{\mathcal{C}_p}_+ \\[5pt]
        & ( X^{(1)}, \ldots,  X^{(p)}) \in \L \cap (\N^{\mathcal{C}_1} \times \cdots \times \N^{\mathcal{C}_p}).
    \end{array}
\end{eqnarray}

\section{McADMM: a multi-clique symmetric Gauss-Seidel proximal ADMM} \label{sec-ADMM} Finally, we are in position to outline the adaptation of \SGSadmm\!\! framework (see \cite{LST2015,STY2015}, and more recently \cite{CST2015}) into the multi-clique setting. We remark that these multi-block \textit{convergent} variants of the ADMM are suitable in our case since the SDPs \eqref{eq.51} and \eqref{eq.3} each involve three blocks of variables, whence a directly extended ADMM is not guaranteed to converge. Furthermore, we adapt the inexact framework of \cite{CST2015} in McADMM where the subproblems in each iteration are allowed to be solved inexactly. While this is clearly advantageous in terms of computational cost, the flexibility of our algorithm will also be improved, for example, this flexibility will allow us to tackle SDPs with large-scale systems of linear equations that can only be solved efficiently using iterative solvers.

\subsection{McADMM for general SDPs of the form \texorpdfstring{\eqref{eq.3}}{(3)}} First, we consider a general multi-block \SGSadmm template applicable to SDPs of the form \eqref{eq.3}. Here, we shall not implement the \textit{linear-space projection} strategy as the second set constraints (which is encoded in the function $\mathscr{G}(X)$ in \eqref{eq.4}) on the SDP can be much more general than $(\N^{\mathcal{C}_1} \times \cdots \times \N^{\mathcal{C}_p})$. As discussed earlier, consistency constraints are represented in terms of linear constraints $\cL(X) = h$ so that each clique variable is updated independently except on the update pertaining to solving the linear system associated with this set of linear constraints.

Given the penalty parameter $\sig >0$, consider the augmented Lagrangian function associated with $\eqref{eq.6}$ defined by
\begin{eqnarray}
    L_\sig(y,\xi,S,Z; X) = \left\{
    \begin{array}{l}
        \inprod{-b}{y} + \inprod{-h}{\xi} +\mathscr{F}^*(-S) + \mathscr{G}^*(-Z)
        \\[5pt]
        +\frac{\sig}{2}\norm{\mathscr{A}^* y + \cL^* \xi + S + Z -\mathscr{C} + \sig^{-1} X}^2 -\frac{1}{2\sig}\norm{X}^2.
    \end{array}\right.
    \label{eq.52}
\end{eqnarray}

In deriving the convergence of the algorithm (see \cite{CST2015}), one requires the following assumption:

\begin{assumption}
    The linear maps $\mathscr{A}$ and $\cL$ are surjective.
\end{assumption}

\medskip

Algorithm~\ref{alg:mcadmm-general} presents McADMM for the general problem \eqref{eq.3}. In this algorithm, as well as in Algorithm~\ref{alg:mcadmm-bounded} presented in the next section, quantities such as $X^{(l),k} \in \S^{\cC_l}$ represent the $\cC_l$ sub-matrix of the variable $X$ at iteration $k$. In the case where the SDP \eqref{eq.6} is obtained from a clique decomposition of a single-variable SDP, the parameters $\mathscr{A}$ and $\mathscr{C}$ are obtained by decomposing (with overlap-compensation) the original parameters $A$ and $C$, in the manner described in Section~\ref{subsec-overlap-compensation}.

\IncMargin{1em}
\begin{algorithm}[tp]
\small
\caption{General McADMM for \eqref{eq.3}}
\label{alg:mcadmm-general}
\SetAlgoLined
\textbf{Input:} $\sigma>0$ and $\tau \in [1,(1+\sqrt{5})/2)$, and a summable sequence of tolerances $\{\varepsilon_k \geq 0 \}$. \\
\textbf{Initialize data:} $(y^0,\xi^0,S^0,Z^0,X^0)$ such that $\xi^0=0$, $\cA(X^0)=b$, and $y^0=(\cA\cA^*)^{-1}\cA(C-S^0-Z^0)$.\\
Find an appropriate clique assignment
$\{\mathcal{C}_1,\ldots \mathcal{C}_p\}$, e.g., using CMD (Algorithm~\ref{algo-CMD}). \\
\BlankLine
Decompose the parameters $C$ and $\cA$ into $\mathscr{C}$ and $\mathscr{A}$, respectively. Keep $y^0$ as a global variable and decompose only the initial matrix variables $(S^0,Z^0,X^0)$ into $(S^{(l),0},Z^{(l),0},X^{(l),0})$ for $l = 1, \ldots, p$.\\
\BlankLine
Perform the following steps at each iteration $k$ until convergence
\BlankLine
\textbf{Step 1a. } Let ${R}_\xi^k =  \mathscr{A}^* y^k +\cL^* \xi^k +  S^k + Z^k - \mathscr{C} + \sig^{-1} X^k.$ Compute
\begin{eqnarray*}
    \widehat{\xi}^{k+1} \approx \mbox{argmin}_{\xi} \;
    L_\sig (y^k,\xi,S^k,Z^k;X^k)
    = \mbox{argmin}_{\xi} \Big\{  \inprod{-h}{\xi} + \frac{\sig}{2}
    \norm{\cL^* (\xi-\xi^k)  + {R}_\xi^k}^2
    \Big\}
\end{eqnarray*}
such that the residual norm $\norm{(h - \sig\cL R^k_\xi) - \sig\cL\cL^* (\widehat{\xi}^{k+1}-\xi^k)} \leq \varepsilon_k.$
\BlankLine
\textbf{Step 1b. } For each clique $l = 1, \ldots p$, compute
\begin{eqnarray*}
    &&\hspace{-7mm}  Z^{(l),k+1} = \mbox{argmin}_{Z^{(l)}} \;
    L_\sig (y^k,\widehat{\xi}^{k+1},S^{(l),k},Z^{(l)};X^{(l),k})
    \\[5pt]
    &=& \mbox{argmin}_{Z^{(l)}} \Big\{ (g^{(l)})^*(-Z^{(l)}) +
    \frac{\sig}{2} \norm{Z^{(l)}-Z^{(l),k} + {R}_Z^{(l),k}}^2 \Big\}
    \;=\; -{\rm Prox}_{\sig^{-1}(g^{(l)})^*}(R_Z^{(l),k}-Z^{(l),k}),
\end{eqnarray*}
where ${R}_Z^{(l),k} =   \left(\cA^{(l)}\right)^*y^k +\left(\cL^{(l)}\right)^*\widehat{\xi}^{k+1} + S^{(l),k} + Z^{(l),k} - C^{(l)} + \sig^{-1} X^{(l),k}$.
\BlankLine
\textbf{Step 1c. }  Let ${R}_\xi^{k+1} = \mathscr{A}^* y^k +\cL^*\xi^k+ S^k + Z^{k+1} - \mathscr{C} + \sig^{-1} X^k$. Compute
\begin{eqnarray*}
    {\xi}^{k+1} \approx \mbox{argmin}_{\xi} \;
    L_\sig (y^k,\xi,S^k,Z^{k+1};X^k)
    = \mbox{argmin}_{\xi} \Big\{  \inprod{-h}{\xi} + \frac{\sig}{2}
    \norm{\cL^* (\xi-\xi^k)  + {R}_\xi^{k+1}}^2
    \Big\}
\end{eqnarray*}
such that the residual norm
$\norm{( h - \sig\cL R^{k+1}_\xi) - \sig \cL\cL^* (\xi^{k+1}-\xi^k)}\leq 10 \varepsilon_k$.
\BlankLine
\textbf{Step 2a. }  Let ${R}_y^k =\mathscr{A}^*  y^k+ \cL^* \xi^{k+1} + S^k + Z^{k+1} - \mathscr{C} + \sig^{-1} X^k$.
Compute \begin{eqnarray*}
    \widehat{y}^{k+1} \approx \mbox{argmin}_y \;
    L_\sig (y,\xi^{k+1},S^k,Z^{k+1};X^k)
    = \mbox{argmin}_y \Big\{  \inprod{-b}{y} + \frac{\sig}{2}
    \norm{\mathscr{A}^* (y-y^k)  + {R}_y^k}^2
    \Big\}
\end{eqnarray*}
such that the residual norm $\norm{( b - \sig\mathscr{A} {R}_y^k) -\sig \mathscr{A}\mathscr{A}^* (\widehat{y}^{k+1}-y^k)} \leq \varepsilon_k.$
\BlankLine
\textbf{Step 2b. }For each clique $l = 1, \ldots p$, compute \begin{eqnarray*}
    && \hspace{-7mm} S^{(l),k+1} = \mbox{argmin}_{S^{(l)}} \;
    L_\sig (\widehat{y}^{k+1},\xi^{k+1},S^{(l)},Z^{(l),k+1} ;X^{(l),k})
    \\[5pt]
    &=& \mbox{argmin}_{S^{(l)}} \Big\{   (f^{(l)})^*(-S^{(l)}) + \frac{\sig}{2}\norm{S^{(l)}-S^{(l),k} + R_S^{(l),k}  }^2
    \Big\}
    \;=\; -{\rm Prox}_{\sig^{-1}(f^{(l)})^*}(R_S^{(l),k}-S^{(l),k}),
\end{eqnarray*}
where $R_S^{(l),k} = \left(\cA^{(l)}\right)^* \widehat{y}^{k+1} + \left(\cL^{(l)}\right)^* \xi^{k+1} + S^{(l),k}  + Z^{(l),k+1} - C^{(l)}+ \sig^{-1} X^{(l),k}$.
\BlankLine
\textbf{Step 2c. }Let ${R}_y^{k+1} = \mathscr{A}^* y^k +  \cL^* \xi^{k+1}+ S^{k+1} + Z^{k+1} - \mathscr{C} + \sig^{-1} X^k$. Compute \begin{eqnarray*}
    {y}^{k+1} \approx \mbox{argmin}_y \;
    L_\sig (y,\xi^{k+1} ,S^{k+1},Z^{k+1} ;X^k)
    = \mbox{argmin}_y \Big\{  \inprod{-b}{y} + \frac{\sig}{2}
    \norm{\mathscr{A}^* (y-y^k)  +  {R}_y^{k+1}}^2
    \Big\}
\end{eqnarray*}
such that the residual norm $\norm{( b - \sig\mathscr{A} R^{k+1}_y) - \sig\mathscr{A}\mathscr{A}^* ({y}^{k+1}-y^k)} \leq 10 \varepsilon_k.$
\BlankLine
\textbf{Step 3. } For each clique $l = 1, \ldots p$, compute \begin{eqnarray*}
    X^{(l),k+1}= X^{(l),k} + \tau \sig \left(\left(\cA^{(l)}\right)^* y^{k+1} +\left(\cL^{(l)}\right)^* \xi^{k+1}+ S^{(l),k+1}+ Z^{(l),k+1}- C^{(l)}  \right).
\end{eqnarray*}
Upon termination, recover original matrices.
\vspace{-2pt}
\end{algorithm}
\DecMargin{1em}

We shall discuss the computational details of Steps \textbf{1b} and \textbf{2b} in the next subsection where a concrete setup is considered. Here, we first outline some practical aspects pertaining to the updates of $\xi^{k+1}$ and $y^{k+1}$ in Steps \textbf{1c} and \textbf{2c}, respectively. These updates correspond to solving some linear systems, where variable decomposition has little effect in terms of computational complexity.

From the optimality condition for the subproblem in Step \textbf{2a}, it is clear that $d_y:=\widehat{y}^{k+1}-y^k$ is the solution to the following system of linear equations:
\begin{eqnarray}
    \sig^{-1} b - \mathscr{A} R^k_y = G\, d_y = \Big(\mbox{$\sum_{l=1}^p$} \cA^{(l)} (\cA^{(l)})^*\Big) d_y.
    \label{eq.53}
\end{eqnarray}
We can easily solve \eqref{eq.53} if the sparse Cholesky factorization of $\mathscr{A}\mathscr{A}^*$ can be computed at a moderate cost and, more importantly, only once during initialization. For the purpose of factorization, we have
\begin{align}
    \mathscr{A}^* 		&= {\begin{bmatrix} \left(\cA^{(1)}\right)^*  \\ \vdots \\ \left(\cA^{(p)}\right)^*\end{bmatrix}} \in \R^{ N_F \times m}, \quad \text{so that }\quad
    {G:=\mathscr{A} \mathscr{A}^*}	= \sum_{l=1}^p \left(\cA^{(l)}\right) \left(\cA^{(l)}\right)^* \in \S^{m}. \label{eq.55}
\end{align}
The dominant operation at this step is therefore given by forward-backward triangular solves based on a precomputed sparse Cholesky factorization of $G$:
\begin{align}\label{eq.56}
    y^{k+1} \leftarrow G^{-1}\left[\sum_{l=1}^p \cA^{(l)} (C^{(l)} - Z^{(l),k+1} - S^{(l),k+1})
    {+\sig^{-1}\bigl(b-\mathscr{A}(X^k)\bigr)}\right],
\end{align}
which retains an $m\times m$ normal-equation system.

Alternatively, \eqref{eq.53} can also be solved using an iterative solver such as the preconditioned conjugate gradient (PCG) method. If a parallel computing platform is available, the matrix-vector operation, $\sum_{l=1}^p \cA^{(l)} (\cA^{(l)})^* d_y$, needed in the PCG method can be done in parallel for any given  $d_y$.

Our approach for $d_\xi:=\widehat{\xi}^{k+1}-\xi^k$ in Step \textbf{1a} is practically identical to the above; we solve the following system of linear equations:
\begin{eqnarray}
    \sig^{-1}h - \cL R^k_\xi =  \cL\cL^* d_\xi.
    \label{eq.57}
\end{eqnarray}
Again, depending on the sparsity structure of $\cL\cL^*$, one may either solve the linear system of equations \eqref{eq.57} via the sparse Cholesky factorization or the PCG method. For the coupling constraints arising from our clique decomposition algorithms (Section~\ref{sec-clique-decomp}), the coefficient matrix $\cL\cL^*$ can often be highly sparse and its Cholesky factorization can be computed without much difficulty. Of course, such a property is contingent on a favorable clique tree structure; see for example, \cite{ZL2021} for a more comprehensive discussion on the sparsity pattern of $\cL\cL^*$.

We reiterate that the subproblems in Steps \textbf{1c} and \textbf{2c} can be solved inexactly, so that for example, we can try to use $\widehat{y}^{k+1}$ to estimate $y^{k+1}$. In this case, the corresponding residual norm (at Step \textbf{2c}) is given by
\begin{eqnarray*}
    \hspace{-0.7cm}
    \norm{( b - \sig\mathscr{A} R^{k+1}_y) - \sig\mathscr{A}\mathscr{A}^* (\widehat{y}^{k+1}-y^k)}
    &\leq& \norm{( b - \sig\mathscr{A} R^{k}_y) - \sig\mathscr{A}\mathscr{A}^* (\widehat{y}^{k+1}-y^k)}
    + \norm{ \sig\mathscr{A} (R^{k+1}_y-R^k_y)}
    \\[5pt]
    &\leq&  \varepsilon_k\, +  \norm{ \sig\mathscr{A} (S^{k+1}-S^k)} .
\end{eqnarray*}
Therefore, we can accept $\widehat{y}^{k+1}$ as an approximate solution in Step \textbf{2c} if $\norm{ \sig\mathscr{A} (S^{k+1}-S^k)} \leq 9\varepsilon_k$. Similarly, we can approximate $\xi^{k+1}$ by $\widehat{\xi}^{k+1}$ if $\norm{ \sig\cL (Z^{k+1}-Z^k)} \leq 9\varepsilon_k$. Due to this inexact setup, we can typically skip the subproblems in Steps \textbf{1c} and \textbf{2c} in practice when the iterates are close to the solution. As a result, the computational savings in each iteration can be significant, especially if the PCG method is used to solve the linear systems.

Finally, the last step in Algorithm~\ref{alg:mcadmm-general} recovers specified entries of the original partial matrix $X$ from the clique variables $(X^{(1)}, \ldots, X^{(p)})$. This can be done cheaply:
\begin{align}
    \text{svec}(X) \leftarrow H \cD \left[\text{svec}\big(X^{(1)}\big)\,;\, \cdots\, ;\, \text{svec}\big(X^{(p)}\big)  \right], \label{eq.58}
\end{align}
where the matrices $H$ and $\cD$ are as defined in Section~\ref{subsec-overlap-compensation}. The corresponding original dual variables are given by ${\operatorname{svec}(S)=\cD\operatorname{svec}_F(S)}$ and ${\operatorname{svec}(Z)=\cD\operatorname{svec}_F(Z)}$.

In Appendix~\ref{sec-ALM}, we discuss another variant of the above \SGSadmm\!\!-type algorithm where the updates are solved using the semismooth Newton-CG method, similar to that proposed in \cite{YST2015}.

\subsection{McADMM for bounded SDPs of the form \texorpdfstring{\eqref{eq.51}}{(51)}} Algorithm~\ref{alg:mcadmm-general} can be directly specialized for bounded SDPs of the type \eqref{eq.2}. In this case, we expand $\mathscr{F}$ according to \eqref{eq.4} in terms of the functions $f^{(l)}(X^{(l)}) = \delta_{\S^{\cC_l}_+}(X^{(l)})$ for $l = 1, \ldots, p$ whereas $\mathscr{G}$ is simply the indicator function  $\delta_{\L \cap (\N^{\cC_1} \times \cdots \times \N^{\cC_p}) }(X)$. That is, we have used the \textit{linear-space projection} approach derived in Section~\ref{sec-linproj} to improve efficiency. This will lighten the computational load since we can replace the update Steps \textbf{1a}--\textbf{1c} of the McADMM in Algorithm~\ref{alg:mcadmm-general} with a single ``average and project'' step given by \eqref{eq.46}.

Let us first specify the augmented Lagrangian for the reformulated problem, i.e., for the dual of \eqref{eq.51} given by:
\begin{align}\label{eq.59}
    d^*&=\max_{y,Z,S} \hspace{0.1cm} \langle b,y \rangle -  \sum_{l=1}^p  \delta^*_{ \S^{\cC_l}_+ }\left(- S^{(l)}\right)
    -  \delta^*_{ \L \cap (\N^{\mathcal{C}_1} \times \cdots \times \N^{\mathcal{C}_p})}\left(- Z^{(1)}, \ldots , - Z^{(p)}\right)\nonumber\\[5pt]
    &\qquad \text{s.t.} \quad S^{(l)} = C^{(l)} - \sum_{i=1}^m y_i A^{(l)}_i - Z^{(l)} \hspace{0.2cm} \text{for } l = 1, \ldots, p .
\end{align}
We have
\begin{align}\label{eq.60}
    L_\sig(y,S,Z; X) =  & \; \inprod{-b}{y}  +\sum_{l=1}^p  \delta^*_{ \S^{\cC_l}_+ }\left(- S^{(l)}\right) \;+\;  \delta^*_{ \L \cap (\N^{\mathcal{C}_1} \times \cdots \times \N^{\mathcal{C}_p})}\left(- Z^{(1)}, \ldots , - Z^{(p)}\right)\nonumber
    \\[5pt]
    &+  \sum_{l=1}^p  \frac{\sig}{2}\norm{(\cA^{(l)})^* y + S^{(l)} + Z^{(l)} -C^{(l)} + \sig^{-1} X^{(l)}}^2 -  \sum_{l=1}^p  \frac{1}{2\sig}\norm{X^{(l)}}^2.
\end{align}

\IncMargin{1em}
\begin{algorithm}[tp]
\caption{McADMM for \eqref{eq.51}}
\label{alg:mcadmm-bounded}
\SetAlgoLined
\textbf{Input:} $\sigma>0$ and $\tau\in [1,(1+\sqrt{5})/2)$, and a summable sequence of tolerances $\{\varepsilon_k \geq 0 \}$. \\
\textbf{Initialize data:} $(y^0, S^0,Z^0,X^0)$ such that ${\cA}(X^0)=b$ and $y^0=(\cA\cA^*)^{-1}\cA(C-S^0-Z^0).$\\

Find an appropriate clique assignment $\{\mathcal{C}_1,\ldots,\mathcal{C}_p\}$, e.g., using CMD (Algorithm~\ref{algo-CMD}). \\

\BlankLine
Decompose the parameters $C$ and $\cA$ into $\mathscr{C}$ and $\mathscr{A}$, respectively. Decompose the initial matrix variables $(S^0,Z^0,X^0)$ into $(S^{(l),0},Z^{(l),0},X^{(l),0})$ for $l = 1, \ldots, p$.\\
\BlankLine
Perform the following steps at each iteration $k$.
\BlankLine

\textbf{Step 1. } For each clique $l = 1, \ldots p$, evaluate
\begin{align*}
    {R}_Z^{(l),k} &= \left(\cA^{(l)}\right)^*y^{k}+ S^{(l),k} + Z^{(l),k} - C^{(l)} + \sig^{-1} X^{(l),k}, \\[7pt]
    Q^{(l),k} &= {R}_Z^{(l),k}-Z^{(l),k}.
\end{align*}
After assembling $Q^k=(Q^{(1),k},\ldots,Q^{(p),k})$, for each clique $l=1,\ldots,p$, evaluate
\begin{align*}
    Z^{(l),k+1} &= \sig^{-1}\Pi_{\N^{\cC_l}}\left(\cM_l(\sig Q^k)\right)-Q^{(l),k},
\end{align*}
where $\cM_l$ is defined in \eqref{eq.44}.

\BlankLine
\textbf{Step 2a. } Let $\widehat{h}^k =\sum_{l=1}^p \cA^{(l)} (C^{(l)} - Z^{(l),k+1} - S^{(l),k}){+\sig^{-1}\bigl(b-\mathscr{A}(X^k)\bigr)}$.
Compute 
$ \widehat{y}^{k+1} \approx {G^{-1}} \widehat{h}^k$ such that
$\sigma \norm{\widehat{h}^k- G\widehat{y}^{k+1}  }  \leq  \varepsilon_k$.
\BlankLine
\textbf{Step 2b. }For each clique $l = 1, \ldots p$, evaluate \begin{align*}
    R_S^{(l),k} &= \left(\cA^{(l)}\right)^* \widehat{y}^{k+1}  + S^{(l),k}  + Z^{(l),k+1} - C^{(l)}+ \sig^{-1} X^{(l),k} \\[7pt]
    S^{(l),k+1} &= \Pi_{\S^{\cC_l}_+}\left(   S^{(l),k}  - R_S^{(l),k}  \right).
\end{align*}
\BlankLine
\textbf{Step 2c. } Let $h^k = \sum_{l=1}^p \cA^{(l)} (C^{(l)} - Z^{(l),k+1} - S^{(l),k+1})+\sig^{-1}\bigl(b-\mathscr{A}(X^k)\bigr)$. Compute 
$y^{k+1} \approx G^{-1} h^k$ such that 
$\sigma \norm{h^k - G y^{k+1} }  \leq 10 \varepsilon_k$.
\BlankLine
\textbf{Step 3. } For each clique $l = 1, \ldots p$, evaluate
\begin{align*}
X^{(l),k+1} &=  X^{(l),k} + \tau \sigma \left(   Z^{(l),k+1} +S^{(l),k+1} +\left(\cA^{(l)}\right)^* y^{k+1} - C^{(l)} \right).
\end{align*}
\BlankLine
Upon termination, recover original matrices.
\end{algorithm}
\DecMargin{1em}

\medskip

We present the McADMM for \eqref{eq.51} in Algorithm~\ref{alg:mcadmm-bounded}. Once again, most updates on the matrix variables are executed clique-wise and independently, with the exception of operations such as \eqref{eq.46} (Step {\bf 1}) where $\cM_l (V)$ (the \textit{averaging} operation presented in \eqref{eq.44}) means that inter-clique operations are still unavoidable. Nonetheless, we find that these are much cheaper compared to the updates pertaining to Steps \textbf{1a}--\textbf{1c} in Algorithm~\ref{alg:mcadmm-general}, particularly when the number of additional constraints $m^+$ is large (say, when $m^+ > m$).

The primary advantage of McADMM shall be reflected in the computation of $S^{k+1}$ in Step {\bf 2b} of Algorithm~\ref{alg:mcadmm-bounded} given a conducive clique distribution, i.e., based on the clique decomposition strategies discussed in Section~\ref{sec-clique-decomp}. More specifically, by computing the spectral decomposition $ S^{(l),k} - R^{(l),k}_S = Q^{(l)} \diag{\lam^{(l)}_1,\ldots,\lam^{(l)}_{n_l}} (Q^{(l)})^\top$, we can compute $S^{(l),k+1}$ as follows:
\begin{eqnarray}
  S^{(l),k+1} =
  Q^{(l)}  \diag{d^{(l)}_1,\ldots,d^{(l)}_{|\cC_l|}}  (Q^{(l)})^\top,
\label{eq.61}
\end{eqnarray}
where $d^{(l)}_j = \max(\lam^{(l)}_j,0) $, $j=1,\ldots,|\cC_l|$. Furthermore, it is clear that the computation of $S^{(l),k+1}$, $l=1,\ldots,p$, can be done in parallel if such a computing platform is available.

The update step for $Z^{k+1}$ in Step {\bf 1} of Algorithm~\ref{alg:mcadmm-bounded} is obtained as follows:
\begin{eqnarray}
  Z^{k+1} &=& -{\rm Prox}_{\sig^{-1} \mathscr{G}^*} (  R^k_Z-Z^k) =
\sig^{-1}{\Pi}_{ \L \cap (\N^{\mathcal{C}_1} \times \cdots \times \N^{\mathcal{C}_p})} \big(\sig ( R^k_Z-Z^k)\big) - ( R^k_Z-Z^k),
\label{eq.62}
\end{eqnarray}
where the projection onto the intersection $\L \cap (\N^{\mathcal{C}_1} \times \cdots \times \N^{\mathcal{C}_p})$ can be computed via Proposition~\ref{prop-intersection}.

\section{Numerical experiments} \label{sec-numerical} In this section, we present numerical experiments on a broad range of large-scale structured SDPs to demonstrate the effectiveness, scalability, and computational advantages of McADMM relative to existing approaches across diverse problem classes and sparsity patterns, as well as its potential for parallel acceleration.

\subsection{Experimental setting}\label{sec-experimental-setting} \textbf{Implementation.} When running McADMM, the user may select their preferred decomposition scheme which includes no decomposition, CMD decomposition, or the direct non-ordered decomposition as described in Section~\ref{sec-maxnorm} (applicable only on arrowhead patterned SDPs). For all settings involving non-trivial decompositions, one may use either the standard linear constraint based approach (as done in Steps \textbf{1a}--\textbf{1c} of Algorithm~\ref{alg:mcadmm-general}) or our proposed \textit{linear-space projection} method as presented in Algorithm~\ref{alg:mcadmm-bounded} --- the latter is set as the default. Finally, as remarked in the last section, we may perform the updates of the variable $y$ using either a direct method (Cholesky factorization) or an iterative approach (PCG). Note that whenever a decomposition returns a single clique, McADMM automatically reduces to sGS-PADMM.
{For the parameter $\delta$ in Condition 1 of CMD (Algorithm~\ref{algo-CMD}), we use}
\[
{\ell:=\max\{|\cC_l|,|\cC_r|\},\qquad
 g(\ell):=\max\{0.11-0.02\lceil\ell/10\rceil,-0.01\},\qquad
 \delta=g(\ell)}.
\]
{For Conditions 3 and 4 of Algorithm~\ref{algo-CMD}, we use the default values $\tau_2=0.5$ and $\tau_3=30$, respectively.}

\medskip
\noindent\textbf{Baseline Solvers.}
To benchmark McADMM's performance, we also run the sparse SDP solver CDCS \cite{ZFP2020} that is based on the SparseCoLO clique decomposition with the variable $X$ being maintained to impose the consistency constraints across different clique variables $X^{(l)}$, $l=1,\ldots,p$.
Furthermore, to analyze the benefit of the decompositions themselves, all experiments are also solved by sGS-PADMM on \eqref{eq.1} without clique decomposition, which can be easily implemented using McADMM by restricting the number of cliques to one.

\medskip
\noindent\textbf{Computational Environment.}
The {\sc Matlab} experiments reported in Tables~\ref{tab:maxcut-sdplib}--\ref{tab:projection-comparison} were run on a machine with the following specifications:
3rd Gen Intel(R) Core(TM) i7-3770, 3401MHz, 4 Cores, 8 Logical Processors.

\medskip
\noindent\textbf{Stopping Conditions.}
Based on the KKT conditions \eqref{eq.9}, we define the following relative KKT residuals to assess the accuracy of the solutions computed by McADMM and sGS-PADMM:
\normalsize
\[
\begin{array}{rlrl}
\eta_{\rm P}
:=&\max\left\{
\frac{\norm{\mathscr A(X)-b}}{1+\norm{b}},
\frac{\norm{\cL(X)-h}}{1+\norm{h}}
\right\},
\quad 
&\eta_{\rm D}
:= &\frac{\norm{\mathscr A^*y+\cL^*\xi+S+Z-\mathscr C}_{\cX}}
{1+\norm{\mathscr C}_{\cX}},
\\[6pt]
\eta_{\mathscr F}
:=&\frac{\norm{X-{\rm Prox}_{\mathscr F}(X-S)}_{\cX}}
{1+\norm{X}_{\cX}+\norm{S}_{\cX}},
\quad 
&\eta_{\mathscr G}
:=& \frac{\norm{X-{\rm Prox}_{\mathscr G}(X-Z)}_{\cX}}
{1+\norm{X}_{\cX}+\norm{Z}_{\cX}},
\\[6pt]
\eta_{\max}
:=&\max\{\eta_{\rm P},\eta_{\rm D},\eta_{\mathscr F},\eta_{\mathscr G}\}.
\end{array}
\]
\normalsize
{For problem \eqref{eq.51}, the second term in $\eta_{\rm P}$ is absent, and the proximal mappings in $\eta_{\mathscr F}$ and $\eta_{\mathscr G}$ are the PSD projection and the projection in \eqref{eq.46}, respectively. A run achieves the required tolerance when $\eta_{\max}<\varepsilon$.}
For the {\sc Matlab} experiments, the stopping tolerance for the SDP solvers (McADMM, sGS-PADMM, CDCS, and cSDP) is set to $\varepsilon = 1 \times 10^{-5}$ with a maximum running time of $10,000$s.
A superscript $\dagger$ indicates termination at the time limit of $10,000$ seconds, whereas a superscript $\ddagger$ indicates termination at the maximum of $10,000$ iterations.

\medskip
\noindent\textbf{Table Notations.}
Results on all datasets are summarized in Tables~\ref{tab:maxcut-sdplib}--\ref{tab:maxnorm}, where \%nnz denotes the density level of the problem in percentage of the total entries $n^2$. Whenever reported, $n$ denotes the dimension of the single-block SDP matrix supplied to the decomposition, $m$ denotes the number of affine equality constraints in that formulation, and $p$ denotes the number of clique blocks after decomposition. For a problem-specific reformulation such as \eqref{eq.65}, $m$ is counted after the reformulation but before clique decomposition. For the run time, ``Solve'' measures only the time taken to solve the converted SDP problem whereas ``Prep'' measures the time spent in performing clique assignments and executing the decomposition. Unless otherwise specified, all computational times reported in this section are measured in seconds. In Tables~\ref{tab:maxcut-sdplib}--\ref{tab:projection-comparison}, a light-gray cell indicates the shortest solve time among the successfully terminated runs for each instance, and ties at the reported precision are highlighted together.

\subsection{Test problem classes} In our numerical experiments, McADMM is tested on three classes of SDPs that include: large-scale MAXCUT problems, Lov\'{a}sz theta problems, and the max-norm matrix recovery problems.

\medskip
\noindent\textbf{MAXCUT problems.}
The relaxation of the MAXCUT problem can be directly decomposed into the multi-clique form of \eqref{eq.51} (note that bound constraints $\N$ are not present in this problem):
\begin{align}\label{eq.63}
    \min_{X^{(1)}, \ldots, X^{(p)}}   \;\; & \sum_{l=1}^p  \langle C^{(l)},X^{(l)} \rangle \nonumber \\[8pt]
    \text{s.t.} \;\; & 1 \;=\;   \sum_{l=1}^p  \langle E^{ii,(l)}, X^{(l)}\rangle \hspace{0.2cm} \text{for } i=1,\ldots,n\nonumber   \\[5pt]
    & ( X^{(1)}, \ldots,  X^{(p)})  \in  \S^{\mathcal{C}_1}_+ \times \cdots \times \S^{\mathcal{C}_p}_+,
    \qquad
    ( X^{(1)}, \ldots,  X^{(p)}) \in \L ,
\end{align}
where, for $1\leq i\leq j\leq n$, we define
\[
{E^{ij}:=}
\begin{cases}
{e_i e_i^\top}, & i=j,\\[3pt]
{\frac{1}{2}(e_i e_j^\top+e_j e_i^\top)}, & i<j,
\end{cases}
\qquad
{E^{ij,(l)}:=}
\begin{cases}
{|\cO_{ij}|^{-1}E_{\cC_l}E^{ij}E_{\cC_l}^\top},
& l\in\cO_{ij},\\[3pt]
{0_{|\cC_l|\times|\cC_l|}}, & l\notin\cO_{ij}.
\end{cases}
\]
Thus, $\langle E^{ij},X\rangle=X_{ij}$. In its non-decomposed form, $C= - \frac{1}{4}(\text{diag}(Be) -B)$ where $B$ is a given weighted adjacency matrix corresponding to some graph. Note that $|B|$ is the aggregate sparsity matrix that is used to derive our clique decomposition. MAXCUT problem instances in our experiments are extracted from the SDPLIB database \cite{BCH1999} {and the G-set collection \cite{DavisHu2011}}.

\medskip
\noindent\textbf{Lov\'{a}sz theta problems.}
Recall that in Appendix~\ref{sec-theta}, we can convert its original SDP formulation \eqref{eq.32} into a sparse format \eqref{eq.33}. Upon clique decomposition, the latter can be written as
\begin{align}\label{eq.64}
\min_{X^{(1)}, \ldots, X^{(p)}}   \;\; & \sum_{l=1}^p  \langle \widehat{C}^{(l)},X^{(l)} \rangle \nonumber \\[8pt]
\text{s.t.} \;\; & 0 \;=\;   \sum_{l=1}^p  \langle A_k^{(l)}, X^{(l)}\rangle \hspace{0.2cm} \text{for } k = 1, \ldots, |\Theta| \nonumber \\[5pt]
		 & 1 =   \sum_{l=1}^p  \langle A_{|\Theta|+1}^{(l)}, X^{(l)}\rangle \nonumber\\[5pt]
	&	( X^{(1)}, \ldots,  X^{(p)})  \in  \S^{\mathcal{C}_1}_+ \times \cdots \times \S^{\mathcal{C}_p}_+, \qquad
		( X^{(1)}, \ldots,  X^{(p)}) \in \L \cap (\N^{\mathcal{C}_1} \times \cdots \times \N^{\mathcal{C}_p}).
\end{align}
In the above, the bounds in $\N^{\cC_l}$ are given by $U^{(l)} = \infty_{|\cC_l| \times |\cC_l|}$ and $L^{(l)} = -\infty_{|\cC_l| \times |\cC_l|}$. Instances of this problem can be retrieved from the Sloane \cite{SLN2005} and DIMACS \cite{JM1996} datasets. We also consider Lov\'{a}sz-theta SDPs constructed from graphs in the G-set collection \cite{DavisHu2011}. The same G-set graph may therefore appear in a MAXCUT formulation and in a distinct Lov\'{a}sz-theta formulation.

\medskip
\noindent\textbf{Max-norm matrix recovery problems.}
A multi-clique form for max-norm matrix recovery problems has been derived in \eqref{eq.28}. {Following \cite{FLT2018}, we also include $\mu\operatorname{trace}(X)$ in the objective to encourage low-rank solutions.} We further convert the inequalities into equality constraints and write $I^{(l)}:=\sum_{i=1}^{d_1+d_2}E^{ii,(l)}$:
\begin{align}\label{eq.65}
    \min_{\substack{X^{(1)}, \ldots, X^{(2K)},\;\theta,\\
    {r^{(1)},\ldots,r^{(2K)},\;s^{(1)},\ldots,s^{(2K)}}}} &\;\; \lambda \theta + {\mu\sum_{l=1}^{2K}\langle I^{(l)},X^{(l)}\rangle} + \sum_{l=1}^{2K}\left(\norm{{r^{(l)}}}_1 + \delta_{\R^{|\cC_l|}_+}(s^{(l)})\right) \nonumber\\[8pt]
    \text{s.t.} \;\; & \widehat{A}_{ij} - {r^{(l)}_{ij}} =   \langle E^{ij,(l)}, X^{(l)}\rangle  \hspace{0.2cm} \text{for all }(i,j) \in{\widehat{\cJ}} \cap (\cC_l\times \cC_l), \text{ and } l=1,\dots,2K  \nonumber \\[5pt]
    & \theta -  s^{(l)}_{i} =    {X^{(l)}_{ii}}  \hspace{0.2cm}\text{for all } i = 1, \ldots, |\cC_l|, \text{ and } l =1,\dots,2K  \nonumber\\[5pt]
    & ( X^{(1)}, \ldots,  X^{(2K)})  \in  \S^{\mathcal{C}_1}_+ \times \cdots \times \S^{\mathcal{C}_{2K}}_+, \;\;
    ( X^{(1)}, \ldots,  X^{(2K)}) \in \L \cap (\N^{\mathcal{C}_1} \times \cdots \times \N^{\mathcal{C}_{2K}}) ,
\end{align}
using ${r^{(l)}} \in \R^{m_l}$ and $s^{(l)} \in \R^{|\cC_l|}$ for $l = 1, \ldots, 2K$ as auxiliary variables, with $m_l = |{\widehat{\cJ}} \cap (\cC_l\times \cC_l)|$. All bound constraints in \eqref{eq.28} have been encoded in the last set constraint. {To obtain a linear objective, we write each residual as $r=r^+-r^-$, with $r^+,r^-\geq0$, and replace $\norm{r}_1$ by $e^\top(r^++r^-)$.}

We generate $M^*=M_LM_R^\top$, where $M_L\in\R^{d_1\times 10}$ and $M_R\in\R^{d_2\times 10}$ have independent entries drawn uniformly from $(0,1)$, and set $\alpha=\lceil\max_{i,j}M^*_{ij}\rceil$. The observed set $\mathcal J$ is generated using non-uniform row and column weights. For both rows and columns, the first 10\%, the next 10\%, and the remaining 80\% receive weights 2, 4, and 1, respectively. The row and column weights are separately normalized to have mean one, giving $p_i$ and $q_j$. Independent variables $u_{ij}\sim U(0,1)$ are generated and assigned scores $\omega_{ij}=u_{ij}p_iq_j$. For a prescribed sampling parameter $\mathrm{SR}$, let $k=\operatorname{round}(\mathrm{SR}d_1d_2)$ and let $\omega_{(k)}$ denote the $k$th-largest score. The observed set is $\mathcal J=\{(i,j):\omega_{ij}>\omega_{(k)}\}$. Writing $\xi\in\R^{|\mathcal J|}$ for a vector with independent $N(0,1)$ entries, the noisy observations are generated as
\[
A_{\mathcal J}=M^*_{\mathcal J}+0.1\norm{M^*_{\mathcal J}}_2\frac{\xi}{\norm{\xi}_2}.
\]
The regularization parameters are
\[
\lambda=\frac{0.8}{\sqrt{|\mathcal J|}}\sum_{(i,j)\in\mathcal J}|A_{ij}|,
\qquad
\mu=\frac{1.25\lambda}{d_1+d_2}.
\]
We use problem sizes $(d_1,d_2)\in\{(300,100),(500,200),(1200,200),(1200,300)\}$ and target sampling parameters $\mathrm{SR}\in\{0.1,0.15,0.2,0.25,0.3,0.35\}$.

\subsection{Effects of decomposition and comparison with CDCS} Tables~\ref{tab:maxcut-sdplib}--\ref{tab:maxnorm} compare McADMM with sGS-PADMM and CDCS on MAXCUT, Lov\'{a}sz theta, and max-norm matrix recovery problems. They report the problem and decomposition statistics, solve and preparation times, and iteration counts. We next summarize the main observations from these comparisons.

\begin{description}
\item [(I)]{\textbf{Effect of decomposition}}

\begin{enumerate}
    \setlength\itemsep{0.3cm}
    \item[(1)] When compared to the single-block sGS-PADMM, it is not difficult to see that McADMM -- and to a lesser extent, other multi-clique solvers -- generally exhibit stronger improvements on sparser SDP instances. Of course, this dynamic is by no means determined based on the sparsity level alone. We remark that for MAXCUT problems, significant accelerations are typically observed on instances of under 2\% density whereas for the Lovász theta instances the threshold is less definitive.
    \item[(2)] Distribution of entries constituting the sparsity pattern plays an equally important role. Consider the seemingly sparse problems G43 (MAXCUT, 2\%) and 1zc.1024 (Lov\'{a}sz theta, 3.5\%): no improvements from sGS-PADMM can be observed due to the non-conducive aggregate sparsity pattern of the original SDP. For these problems, the aggregate sparsity patterns become dense upon chordal completion and they lead to a large number of dense or overly-large cliques.
    \item[(3)] Instance size is another vital factor; this is well demonstrated in Table~\ref{tab:maxnorm} where the arrowhead decomposition scheme shows a significant acceleration on large problems almost independently of the density level. On the other hand, it barely shows any improvements on the smallest instances even for the low density level of 4\%.
\end{enumerate}

\item[(II)] \textbf{Comparison against CDCS}
\begin{enumerate}
\setlength\itemsep{0.3cm}
\item[(1)] First, we remark that in CDCS, only the PSD projection step is decomposed throughout the algorithm. More precisely, CDCS stores the original variable $X$ and links it with the clique variables $X^{(l)}$ ($l\in[p]$) by adding the \textit{consistency} linear constraints
$E_{\cC_l} X E_{\cC_l}^\top = X^{(l)}$ as in Theorem~\ref{thm-1}
for each $l\in [p].$

\item[(2)] McADMM can be seen to dominate CDCS in general; this is a consequence of McADMM's approach to handle the clique variables directly without relying on the original variable $X$ to maintain the consistency constraints described in \eqref{eq.25}. With a more exhaustive decomposition and using the \textit{linear-space projection} approach in Section~\ref{sec-linproj}, McADMM's subproblems are more efficiently solved and require less memory.

\item[(3)] McADMM often requires longer preparation times, especially on large problems with a large number of constraints. This reflects our slightly more involved conversion procedure which includes running CMD, extracting the clique parameters from $\cA$ and $C$ and compensating the overlaps. Nonetheless, this drawback is typically insignificant compared to the savings from solving the resulting converted SDP.

\item[(4)] McADMM appears to be more robust against elevated density level and aggregate pattern distribution. Consider for instance the MAXCUT problems G43–47; while these instances correspond to unfavourable sparsity patterns, McADMM is not penalized as much as CDCS and still executes the problems within reasonable time.
\end{enumerate}
\end{description}

\setcounter{table}{0}
\renewcommand{\thetable}{6.\arabic{table}}

\begin{footnotesize}
\begin{longtable}{ |c|r|r||    r|r|r|r||     r|r||     r|r|r|}
\caption{Comparison of McADMM, sGS-PADMM, and CDCS on some MAXCUT problems from SDPLIB and the G-set collection.}\label{tab:maxcut-sdplib}\\
\hline
\hline
 \multicolumn{3}{|c||}{}					& \multicolumn{4}{c||}{McADMM} 				&\multicolumn{2}{c||}{sGS-PADMM} 	& \multicolumn{3}{c|}{CDCS}	\\
\hline
Problem 	& \%nnz 			&$n$		&$p$ 	& Solve    			& Prep  	&Iter		& Solve 			&Iter		& Solve		  	&Prep 			&Iter\\
\hline
\endfirsthead
\hline
\hline
 \multicolumn{3}{|c||}{}					& \multicolumn{4}{c||}{McADMM} 				&\multicolumn{2}{c||}{sGS-PADMM} 	& \multicolumn{3}{c|}{CDCS}	\\
\hline
Problem 	& \%nnz 			&$n$		&$p$ 	& Solve    			& Prep  	&Iter		& Solve 			&Iter		& Solve		  	&Prep 			&Iter\\
\hline
\endhead
G43 		& 2.1 		    		&1000    	& 101 	& 437.4 			&2.8 		&1004 	&\best{150.3}			&1110	&6988.4 			&2.0		&10000$^{\ddagger}$\\

G44 		& 2.1 		   		&1000    	& 97 		& 409.5 			&2.7 		&1004 	&\best{165.6}			&1193	&6812.2 			&1.5		&10000$^{\ddagger}$\\

G45 		& 2.1 		  		&1000    	& 97 		& 263.9 			&2.5 		&915		&\best{166.6}			&1210	&7752.1 			&1.6		&10000$^{\ddagger}$\\

G46 		& 2.1 		   		&1000    	& 105 	& 470.5 			&2.9 		&1020 	&\best{165.6}			&1216	&7059.9 			&1.4		&10000$^{\ddagger}$\\

G47 		& 2.1 		    		&1000    	& 99 		& 377.1 			&2.6 		&971 	&\best{169.7}			&1201	&7000.4 			&1.6		&10000$^{\ddagger}$\\
\hline
G51 		& 1.3 		    		&1000   	& 77 		& \best{119.3}			&3.1 		&907 	&223.7			&1606	&2223.4 			&0.7		&10000$^{\ddagger}$\\

G52 		& 1.3 		   		&1000   	& 81 		& \best{115.1} 			&3.5 		&908 	&234.0			&1691	&2003.2			&0.7		&10000$^{\ddagger}$\\

G53 		& 1.3 		   		&1000    	& 81 		& \best{119.1} 			&3.5 		&903 	&196.5			&1511	&2278.5 			&0.7		&10000$^{\ddagger}$\\

G54 		& 1.3 		  		&1000    	& 76 		& \best{104.9} 			&3.4 		&815 	&222.7			&1615	&2497.6 			 &0.8		&10000$^{\ddagger}$\\

G55		& 0.12		 	&5000    	& 728 	& \best{2826.4}		 	&146.9 	&1273 	&10000$^\dagger$	&553		&10000$^\dagger$	&24.8		&2184\\
\hline
G56		& 0.12	 	 	&5000    	& 728 	& \best{2834.4}		 	&147.7 	&1296 	&10000$^\dagger$	&550		&10000$^\dagger$ 	&24.7		&2216\\

G57		&  0.1		   		&5000    	& 319 	& \best{496.2} 		 	&82.9 	&2167 	&10000$^\dagger$	&540		&2172.9 			 &10.5	&6931\\

G32		& 0.25		    	&2000    	& 149 	& \best{95.2} 			&14.7 	&1161 	&3667.3			&3698	&331.1			 &1.7		&3475\\

G33		& 0.25	 	   	&2000    	& 132 	& \best{133.6} 			&12.1 	&1720 	&3727.6			&3709	&387.6 			 &1.6		&3964\\

G34		& 0.25	 	    	&2000    	& 163 	& \best{171.8} 			 &13.6 	&1926 	&4063.4			&4081	&416.1 			&1.7		&3778\\
\hline
G35		& 0.64		   	&2000    	& 183 	& \best{470.8} 			 &16.2 	&1068 	&2088.4			&2015	&8303.5 			&3.1		&10000$^{\ddagger}$\\

G36		& 0.64	 	   	&2000    	& 186 	& \best{572.3} 			 &16.4 	&1220 	&2134.8			&2035	&7355.2			&3.0		&10000$^{\ddagger}$\\

G61		& 0.08 	    		&7000    	& 1039	& \best{7353.2}			 &292.8 	&1408 	&10000$^\dagger$	&199		&10000$^\dagger$ 	&52.6		&999	\\

G62		& 0.07	 	   	&7000    	& 438 	& \best{992.7} 			 &156.1 	&2624 	&10000$^\dagger$	&197		&3936.8 			&20.9		&10000$^{\ddagger}$\\

G63		& 0.18	 	    	&7000    	& 725 	& 10000$^\dagger$ 	&240.4 	&1319 	&10000$^\dagger$	&198		&10000$^\dagger$	&59.9		&429	\\
\hline
G65		& 0.06	 	   	&8000    	& 607 	& \best{1450.9} 			 &247.6 	&2750 	&10000$^\dagger$	&129		&6917.9 			&27.4		&10000$^{\ddagger}$\\

G66		& 0.06		   	&9000    	& 806 	& \best{1862.7} 			 &335.6 	&2929 	&10000$^\dagger$	&91		&10000$^\dagger$	&36.5		&7657\\
\hline
\hline
\end{longtable}
\end{footnotesize}

\begin{footnotesize}
\begin{longtable}{|c|r|r||    r|r|r|r||     r|r||     r|r|r|}
\caption{Comparison of McADMM, sGS-PADMM, and CDCS on some Lovász theta problems from Borchers' collection.}\label{tab:theta-borchers}\\
\hline
\hline
 \multicolumn{3}{|c||}{}					& \multicolumn{4}{c||}{McADMM} 					&\multicolumn{2}{c||}{sGS-PADMM} & \multicolumn{3}{c|}{CDCS}			\\
\hline
Problem 		& \%nnz 		&$m$		&$p$ 	  & Solve    		& Prep  			&Iter		& Solve 	 	&Iter	& Solve		  		&Prep 	&Iter\\
\hline
\endfirsthead
\hline
\hline
 \multicolumn{3}{|c||}{}					& \multicolumn{4}{c||}{McADMM} 					&\multicolumn{2}{c||}{sGS-PADMM} & \multicolumn{3}{c|}{CDCS}			\\
\hline
Problem 		& \%nnz 		&$m$		&$p$ 	  & Solve    		& Prep  			&Iter		& Solve 	 	&Iter	& Solve		  		&Prep 	&Iter\\
\hline
\endhead
1dc.64 		& 30.3 		&544	    	& 2 		  & 3.5 		&0.1 				&3657 	& \best{2.7}			&1809&33.2 			 	&0.3		&10\\

1et.64 		& 17.1 		&265  	& 2 		  & 0.8 		&0.0  				&842		& \best{0.5} 			&399	&0.6 			 		&0.0		&362	\\

1tc.64 		& 13.7		&193 	& 2 		  & 1.1 		&0.0				&1081	& \best{0.6}			&395	&2.3 			 		&0.0		& 1190 \\

1dc.128 		& 20 			&1472  	& 4 		  & 7.1 		&0.0				&3050 	& \best{4.2}			&1724&117.4 		 		&0.0		&10000$^{\ddagger}$\\

1et.128 		& 10.4 		&673 	& 3 		  & \best{1.0}		&0.0				&716		& 1.1			&426	&2.1 			 		&0.1		&575	\\
\hline
1tc.128		& 8.5 			&513	    	& 4 		  & \best{1.7} 		&0.0	 			&1269	& 2.8 			&1131&2.0			 	 	&0.0		&505		\\

1zc.128		& 15.8 		&1121 	& 3 		  & 3.6 		&0.2  				&3021	& \best{1.5} 			&256&44.5			 		&0.3		&3751\\

1dc.256		& 12.8 		&3840	 & 16 	 & 79.5 		&0.1				&5595	& \best{19.4}		&2826&362.3 		 		&0.1		&10000$^{\ddagger}$\\

1et.256 		& 6.2 		  	&1665	 & 6 		  & \best{2.2} 		&0.1				&707 	& 7.6			&1211&18.3 			 	&0.1		&1837\\

1tc.256		& 5.1			&1313	& 6 		  & \best{3.2} 		&0.1				&1331 	& 14.9		&2083&11.7 		 	 	&0.0		&1236\\
\hline
1zc.256		& 9.7	 		&2817	& 16 		  & 21.9		&0.1 				&1846 	& \best{1.9}			&317	&75.5 			 	&0.1		&1851\\

1dc.512		&  8			&9728	& 31 		  & 161.4 		&0.3 				&3285 	& \best{86.7}		&2449&2079.2 		 		&0.4		&10000$^{\ddagger}$\\

1et.512		& 3.6		 	&4033	& 14 		  & \best{7.3} 		&0.2	 			&790 	& 42.7		&1359&60.4			 	&0.2		&1741\\

1tc.512		& 3.1	 	 	&3265	& 13 		  & \best{18.8} 		&0.3 				&2028	& 79.3		&2540&99.5 			 	&0.2		&3811\\

2dc.512		& 42	 	  	 &54896	& 6 		  & 206.3 		&0.2 				&3692	& \best{102.0}		&2938&5000.3 		 		&0.7		&10000$^{\ddagger}$\\
\hline
1zc.512		& 5.9		   	&6913 	& 37 		  & 53.1 		&0.4 				&5288	& \best{30.5}		&991	&131.9 		 		&0.4		&633	\\

1dc.1024		& 4.9	 	   	&24064	& 57 		  & \best{560.8} 		&1.2 				&1863	& 592.6		&2909&3876.0		 	 	&2.2		&10000$^{\ddagger}$\\

1et.1024		& 2.1 	 		&9601	& 36	 	  & \best{68.5} 		&1.1  				&1391	& 284.7 		&1440&398.9		 		&0.7		&2906\\

1tc.1024		& 1.8	 	   	&7937	& 36		  & \best{56.0} 		&1.2 				&1480 	& 599.6		&3042 &240.8			 	&0.6		&2815	\\

1zc.1024		& 3.5	 	    	&16641 	& 76 		  & 1573.9		 &1.5 				&4927 	& \best{366.9}		&1690 &2700.2		 	 	&1.7		&2908			\\
\hline
2dc.1024		& 32.5	 	&169163 	& 8 		  & 1054.3 		 &0.8 				&2360 	& \best{478.2}		&2241&10000$^\dagger$ 		&3.5		&4193\\

1dc.2048		& 2.9		   	 &58368 	& 108 	 & \best{4549.9} 		 &5.5 				&3022 	& 5967.6		&4155&10000$^\dagger$ 		&10.0		&1462\\

1et.2048		& 1.2	   	    	&22529	& 99 		  & \best{432.7} 		&4.5 				&2388 	& 4123.0		&2934&799.9	 	 		&2.5		&1681\\

1tc.2048		& 1.0		   	&18945 	& 93 		  & \best{215.6} 		&5.2 				&1650 	& 4105.3		&2970&714.5		 	 	&2.2		&2922	\\

1zc.2048		& 1.0		   	&39425	& 157 	  & \best{1980.6} 		 &7.0 				&1251	& 2649.8		&1757&7879.8 		 		&10.4		&1075\\
\hline
2dc.2048		& 24.0		&504452  	& 12 		  & 9066.2 		 &5.4 				&3347	& \best{3779.3}		&2552&10000$^\dagger$ 		&21.8		&450	\\
\hline
\hline
\end{longtable}
\end{footnotesize}

\begin{footnotesize}
\begin{longtable}{|c|r|r||    r|r|r|r||     r|r||     r|r|r|}
\caption{Comparison of McADMM, sGS-PADMM, and CDCS on Lovász theta problems derived from DIMACS and G-set graph instances.}\label{tab:theta-dimacs-gset}\\
\hline
\hline
 \multicolumn{3}{|c||}{}								& \multicolumn{4}{c||}{McADMM} 							&\multicolumn{2}{c||}{sGS-PADMM} 		& \multicolumn{3}{c|}{CDCS}			\\
\hline
Problem 			& \%nnz 		&$(n,m)$		&$p$ 	  & Solve    		& Prep  			&Iter	& Solve 	 	&Iter	& Solve		  &Prep 	&Iter	\\
\hline
\endfirsthead
\hline
\hline
 \multicolumn{3}{|c||}{}								& \multicolumn{4}{c||}{McADMM} 							&\multicolumn{2}{c||}{sGS-PADMM} 		& \multicolumn{3}{c|}{CDCS}			\\
\hline
Problem 			& \%nnz 		&$(n,m)$		&$p$ 	  & Solve    		& Prep  			&Iter	& Solve 	 	&Iter	& Solve		  &Prep 	&Iter	\\
\hline
\endhead
MANN-a27  		& 1.8 		    &(379,703)	    		& 15 		  & \best{7.6} 			 &0.7 				&729 & 24.6		&1669&12.8 			 &0.2		&838	\\

johnson8-4-4 	& 26.4 		&(71,561)  			& 1 		  & \best{0.2} 			 &0.0  			&100	& \best{0.2} 			&100	&1.3 			 &0.1		&241\\

johnson16-2-4 	& 25.4 		&(121,1681)  		& 4 		  & 1.8 			 &0.0				&487	& \best{0.1}			&64&2.4 			 &0.1		&355\\

san200-0.7-1	& 31.0 		&(201,5971)  	 	& 15 		  & 13.4 			 &0.0				&901& \best{4.3}			&798	&663.7		 &0.1		&10000$^{\ddagger}$\\

sanr200-0.7 	& 31.3 		&(201,6033) 		& 14 	 	  & 38.2 			 &0.1				&1785& \best{2.1}		&361	&95.5 			 &0.2		&2595\\
\hline
c-fat200-1		& 92.4 		&(201,18367)	    	& 1 		   & \best{2.2}			 &0.0  			&363	& \best{2.2} 			&363 &170.3		 &0.1		&10000$^{\ddagger}$\\

hamming-6-4  	& 66.7 		&(65,1313) 	 	    	& 1 		   & \best{0.1}			 &0.0  			&62& \best{0.1} 			&62&19.0			 &0.0		&7710\\

hamming-8-4  	& 36.8 		&(257,11777)	   	& 16 		  & 88.0 			 &0.1				&2136& \best{0.9}		&113	&43.3 			 &0.1		&628\\

hamming-9-8 	& 2.3 		  	&(513,2305)		& 64 		  & \best{4.4} 			 &1.3				&165 & 50.5		&1825&17.7 			 &0.2		&549\\

hamming-10-2	& 4.7			&(1025,23041)	   	& 64 		  & \best{70.0} 			 &0.9				&347& 93.3		&437	&241.7 		 &1.1		&612	\\
\hline
hamming-7-5-6	& 23.9	 	&(129,1793)	   	& 8 		  & 1.7			&0.0 				&225& \best{0.9}			&361&15.7 			 &0.0		&1410\\

hamming-8-3-4	&  50.0		 &(257,16129)	   	& 2 		  & 7.6 			 &0.0 				&641 & \best{1.0}			&136&88.8 			 &0.1		&1512\\

hamming-9-5-6	& 41.4		  &(513,53761)	  	 & 4 		  & 39.1 			 &0.1	 			&842 & \best{21.8}		&699	&1369.7		 &0.7		&3937\\

brock200-1		& 26.6	 	 &(201,5067)	  	 & 15 	 & 33.9 			 &0.1 				&1501& \best{2.5}		&436	&92.2 			 &0.1		&2559\\

brock200-4		& 35.2	 	 &(201,6812)	 	& 13 		  & 32.3 			 &0.0 				&1467& \best{1.8}		&310	&60.1 			 &0.1		&1976\\
\hline
brock400-1		& 25.7		 &(401,20078) 		& 19 		  & 184.2 			 &0.1 				&1873& \best{7.6}		&368	&378.9 		 &0.3		&2119\\

keller4		& 36.2	 	 &(172,5101)		 & 7 		  & 18.3 			 &0.0 				&1704& \best{1.3}		&417	&124.6		 &0.1		&3561\\

p-hat300-1		& 75.9 	 	 &(301,33918)	   	& 1	 	   & \best{8.1}			 &0.0  			& 619& \best{8.1} 		&619	&742.8		 &0.2		&10000$^{\ddagger}$\\

G43			& 2.3	 	   	 &(1001,9991)		& 93		  & \best{169.1} 			 &2.5 				&423 & 260.8		&1655 &2047.1		 &1.5		&3000\\

G44			& 2.3	 	    	 &(1001,9991) 		& 100 	 & \best{112.3}			 &2.6 				&366 & 259.6		&1653 &1954.7		 &1.5		&2867\\
\hline
G45			& 2.3	 	   	 &(1001,9991) 		& 97 		  & \best{124.5} 			 &2.5 				&447 & 256.8		&1643 &2616.6 		 &1.6		&3349\\

G46			& 2.3		   	 &(1001,9991) 		& 100 	  & \best{164.0} 			 &2.8 				&369& 251.3		&1600 &2069.2 		 &1.5		&3025\\

G47			& 2.3	   	    	 &(1001,9991)		& 96 		  & \best{124.8} 			 &2.6 				&354 & 258.3		&1656 &2141.4		 &1.5		&3039\\

G51			& 1.5		   	 &(1001,5910) 		& 62 		  & \best{92.4} 			 &2.9 				&770 & 526.5		&3299 &424.2 		 &0.8		&1845\\

G52			& 1.5		   	 &(1001,5917)		& 58 		  & \best{311.2} 			 &2.8 				&2633& 766.1		&4965 &565.0		 &0.7		&2757\\
\hline
G53			& 1.5		   	 &(1001,5915)  		& 63 		  & \best{467.2} 			 &2.9 				&3669& 641.4		&4269 &721.1 		 &0.8		&3054\\

G54			& 1.5	   	    	 &(1001,5917)		& 65 		  & \best{162.6} 			 &2.9 				&1304& 992.2		&6348 &1118.8 		 &0.8		&4338\\
\hline
\hline
\end{longtable}
\end{footnotesize}

\begin{footnotesize}
\begin{longtable}{ |c r|r|r||    r|r|r|r||     r|r||     r|r|r|}
\caption{Comparison of McADMM, sGS-PADMM, and CDCS on some synthetic max-norm matrix recovery problems.}\label{tab:maxnorm}\\
\hline
\hline
{} 			&{}		&{} 		& {}		 & \multicolumn{4}{c||}{McADMM} 			&\multicolumn{2}{c||}{sGS-PADMM} 						& \multicolumn{3}{c|}{CDCS}			\\
\hline
$(d_1,d_2)$ 	&SR		& \%nnz  		& $m$ 		&$p$ 		& Solve 	&Prep		&Iter	& Solve 		&Iter	& Solve				& Prep	 	&Iter\\
\hline
\endfirsthead
\hline
\hline
{} 			&{}		&{} 		& {}		 & \multicolumn{4}{c||}{McADMM} 			&\multicolumn{2}{c||}{sGS-PADMM} 						& \multicolumn{3}{c|}{CDCS}			\\
\hline
$(d_1,d_2)$ 	&SR		& \%nnz  		& $m$ 		&$p$ 		& Solve 	&Prep		&Iter	& Solve 		&Iter	& Solve				& Prep	 	&Iter\\
\hline
\endhead
(300,100) 	&0.1		& 4.0 		   	& 3399   	&6 			& \best{19.3}	&0.0			&1020&19.7		&1044 &34.7				& 0.5			&1419\\

{}		 &0.15	& 5.9		   	& 4899   	&6 			& \best{15.7}	&0.0			&788&20.0			&935	&30.8 				& 0.2			&826	\\

{}		&0.2		& 7.5 		   	& 6399   	&6 			& \best{14.3}	&0.0			&701&17.8			&805	&42.4 				& 0.2			&914	\\

{}		&0.25		& 9.6 		   	& 7899   	&6 			& \best{14.1}	&0.0			&662	&17.3			&712	&75.8 				& 0.2			&1136\\

{}		&0.3		& 11.5 		& 9399   	&6 			& 14.0	&0.0			&539&\best{12.3}			&579	&182.3				& 0.3			&1886\\

{}		&0.35		& 13.4 		& 10899  	&6 			& \best{12.9}	&0.0			&493&17.6			&701	&250.5				& 0.2			&2046\\
\hline
(500,200) 	&0.1		& 4.2 		   	& 10699  	&5 			& \best{74.5}	&0.0			&1216&122.1		&1476&127.4				& 0.4			&1874\\

{}		&0.15		& 6.3 		   	& 15699  	&5 			& \best{61.3}	&0.0			&988&103.2		&1272&830.2				& 0.5			&3584\\

{}		&0.2		& 8.3 		   	& 20699  	&5 			& \best{64.8}	&0.0			&1001&109.0		&1230&403.6 				& 0.5			&2309\\

{}		&0.25		& 10.3 		& 25699   	 &5			& \best{56.8} 	&0.0			&820&80.8			&910	&611.6 				& 0.7			&2355\\

{}		&0.3		& 12.4		& 30699     	&5 			& \best{44.7}	&0.0			&640	&68.4			&828	&3344.8 				& 0.9			&6102\\

{}		&0.35		& 14.2	 	& 35699  	 &5 			& \best{37.3}	&0.0			&530&60.2			&727	&3903.6 				& 1.0			&6124\\
\hline
(1200,200)	&0.1		& 2.5		   	& 25399   	&12			& \best{247.0} 	&0.0			&1651&943.6		&2239&401.4 				& 1.4			&1857\\

{}		&0.15		& 3.7 		   	& 37399   	&12 			& \best{218.9}	&0.0			&1407&1000.8		&2274&3347.5 				& 1.8			&6153\\

{}		&0.2		& 5.0 		   	& 49399 	 &12 			& \best{187.0}	&0.0			&1238&1026.9		&2257&8281.7 				& 2.2			&10000$^{\ddagger}$\\

{}		&0.25		& 6.2 		   	& 61399   	&12 			& \best{121.4}	&0.0			&741&662.9		&1449&10000$^\dagger$ 		& 2.5			&8828\\

{}		&0.3		& 7.4 		   	& 73399   	&12			& \best{96.3} 	&0.0			&601&676.8		&1293&10000$^\dagger$ 		& 2.9			&7418\\

{}		&0.35		& 8.6 		   	& 85399   	&12 			& \best{101.5}	&0.0			&597&676.2		&1274& 10000$^\dagger$		& 3.2			&6394\\
\hline
(1200,300)	&0.1		& 3.3 		   	& 37499   	&8 			& \best{379.2}	&0.0		 	&1698&1304.1		&2599&784.8 				& 1.7			&2368\\

{}		&0.15		& 4.9 		   	& 55499   	&8 			& \best{337.9}	&0.0		 	&1420&1352.2		&2531&10000$^\dagger$ 		& 2.7			&9093\\

{}		&0.2		& 6.5 		   	& 73499   	&8			& \best{334.1}  	&0.0		 	&1401&1551.9		&2739&10000$^\dagger$ 		& 3.4			&5958\\

{}		&0.25		& 8.1 		   	& 91499   	&8 			& \best{224.7} 	&0.0		 	&905&974.4		&1642&10000$^\dagger$ 		& 4.1			&4549\\

{}		&0.3		& 9.7 		   	& 109499    &8 			& \best{189.1} 	&0.0		  	&710	&477.6		&866	&10000$^\dagger$ 		& 4.7			&3789\\

{}		&0.35		& 11.3 		& 127499   	&8 			& \best{195.6}	&0.0	  		&747	&702.4		&1235&10000$^\dagger$ 		& 5.5			&3013\\
\hline
\hline
\end{longtable}
\end{footnotesize}

\subsection{Benefits of the linear-space projection method} On top of comparing McADMM with other SDP solvers, we also dedicate a separate set of experiments on some MAXCUT and Lov\'{a}sz theta problems to showcase the distinction between our \textit{linear-space projection} approach presented in Section~\ref{sec-linproj} and the standard linear constraints approach \eqref{eq.25}. The numerical results are summarized in Table~\ref{tab:projection-comparison}.

To facilitate a smoother analysis, we select a number of MAXCUT and Lovász theta instances that represent a wide range of overlap intensity. The latter is a quantity that we can describe in terms of the \textit{link ratio}, defined as
\begin{align}
    \text{Link ratio }(\%) = 100\frac{m^+}{m}, \label{eq.66}
\end{align}
where $m^+ =  \sum_{(\cC_l,\cC_r) \in \cF} |\cC_{lr}|(|\cC_{lr}|+1)/2$ is the number of overlap constraints as discussed in Section~\ref{subsec-2.2}. We also note that some possibly non-trivial amount of preparation time must be set aside to construct the linear constraints in the standard approach -- the time required to construct these linking constraints is reported separately under ``Link Prep.''

\begin{enumerate}
\setlength\itemsep{0.3cm}
\item[(1)] The effect of overlap intensity (Link ratio) on the efficacy of the \textit{linear-space projection} method can be smoothly observed in Table~\ref{tab:projection-comparison}. Note that for instances with very few overlaps -- most notably, SDPs whose conversion returns fewer than 5 cliques -- the two approaches hardly differ in their performance since the consistency constraints require negligible resources. In fact, we found that the standard approach at times slightly outperforms our \textit{linear-space projection}. Nonetheless, this occurrence is infrequent and the \textit{linear-space projection} approach remains the favored option for applications within a general-purpose solver.
\end{enumerate}

\begin{footnotesize}
\begin{longtable}{ |l|r|r|r||    r|r|r||     r|r|r|r|}
\caption{Comparison of McADMM performance: using the standard linear-system approach in Algorithm~\ref{alg:mcadmm-general} versus using our \textit{linear-space projection} method to handle the linking constraints, where Steps 1a--1c in Algorithm~\ref{alg:mcadmm-general} are replaced by Step 1 in Algorithm~\ref{alg:mcadmm-bounded}.}\label{tab:projection-comparison}\\
\hline
\hline
 \multicolumn{4}{|c||}{}											& \multicolumn{3}{c||}{McADMM (Algorithm~\ref{alg:mcadmm-bounded})} 	& \multicolumn{4}{c|}{McADMM-Link (Algorithm~\ref{alg:mcadmm-general})}			\\
\hline
Problem 		& \%nnz &$p$ 	&Link ratio (\%)		& Solve    	& Prep  	&Iter	&Solve 	&Prep	&Link Prep.	&Iter\\
\hline
\endfirsthead
\hline
\hline
 \multicolumn{4}{|c||}{}											& \multicolumn{3}{c||}{McADMM (Algorithm~\ref{alg:mcadmm-bounded})} 	& \multicolumn{4}{c|}{McADMM-Link (Algorithm~\ref{alg:mcadmm-general})}			\\
\hline
Problem 		& \%nnz &$p$ 	&Link ratio (\%)		& Solve    	& Prep  	&Iter	&Solve 	&Prep	&Link Prep.	&Iter\\
\hline
\endhead

$\theta$: 1et.128 	& 10.4 & 3 		&1			& 1.0		&0.0		&716	& \best{0.8}		&0.0		&0.0		&595	\\

$\theta$: 1zc.128	& 15.8 & 3 		&156			& \best{3.6} 		&0.2  		&3021& 6.6 	&0.0		&0.0		&2165\\

$\theta$: 1tc.256	& 5.1	& 6 		&14			& \best{3.2} 		&0.1		&1331 & \best{3.2}	&0.1		&0.0		&1027\\

$\theta$: 1et.256 	& 6.2 	& 6 		&56			& \best{2.2} 		&0.1		&707 & \best{2.2}		&0.1		&0.0		&636	 \\

$\theta$: 1tc.512	& 3.1	 & 13 	&308			& \best{18.8} 	  &0.3 	&2028& 20.3	&0.3		&0.0		&1721\\
\hline
$\theta$: san200-0.7-1& 31.0 	 & 15 	&464			& \best{13.4} 	&0.0		&901 & 37.4	&0.1		&0.1		&1929\\

$\theta$: MANN-a27& 1.8 & 15 		&652  		& \best{7.6} 		 &0.7 		&729 & 24.9	&0.7		&0.1		&2317\\

$\theta$: brock400-1& 25.7& 19 		&1361  		& \best{184.2} 	&0.1 		&1873& 248.9	&0.1		 &2.8		&1685\\

$\theta$: G45		& 2.3	& 97 		&6426  		& \best{124.5} 	&2.5 		&447 & 206.6	&2.5		 &7.0		&646	\\

$\theta$: G47		& 2.3	 & 96 		&9356		& \best{124.8} 	&2.6 		&354 & 209.8	&2.6		 &11.5	&482	\\

$\theta$: G43		& 2.3	 & 93		&10933  		& \best{169.1} 	&2.5 		&423 & 320.1	&2.6 	 	 &15.2	&631	\\

$\theta$: G46		& 2.3& 100 	&12339  		& \best{164.0} 	&2.8 		&369& 314.3	&2.7		 &17.4	&550	\\
\hline
MC: G57		&  0.1		& 319 	&3148		& \best{496.2} 	&82.9 	&2167 &898.0	&82.2		&4.4		&5441\\

MC: G36		& 0.64& 186 	&54086		& \best{572.3} 	&16.4 	&1220&916.6	&16.4		&17.7		&1651\\

MC: G55		& 0.12& 728 	&78149		& \best{2826.4}	&146.9 	&1273 &4007.9	&158.7	&69.0		&1607\\

MC: G44 		& 2.1 	& 97 		&103067		& \best{409.5}	&2.7 		&1004 &604.7	&2.6		&13.6		&1206\\

MC: G43 		& 2.1 & 101 	&109588		& \best{437.4} 	&2.8 		&1004 &598.3	&2.8		&14.7		&1140\\

MC: G61		& 0.08 & 1039	&112660		& \best{7353.2}	&292.8 	&1408 &9081.9	&295.2	&152.4	&1601\\

MC: G46 		& 2.1 	& 105 	&122323		& \best{470.5}	&2.9 		&1020 &702.9	&2.9		&17.1		&1195\\
\hline
\hline
\end{longtable}
\end{footnotesize}

\subsection{Julia CPU and GPU implementation}\label{sec:julia-gpu} In addition to the {\sc Matlab} implementation used in Tables~\ref{tab:maxcut-sdplib}--\ref{tab:projection-comparison}, McADMM is also implemented in Julia \cite{Bezanson2017} with CPU and CUDA-enabled GPU execution through CUDA.jl \cite{Besard2019}. Clique construction and sparse-SDP conversion are performed on the CPU. When the aggregate sparsity pattern does not yield a useful multi-clique decomposition, the implementation retains the original single-block problem and follows the sGS-PADMM approach. Otherwise, it uses the McADMM linear-space projection on the converted clique blocks.

The CPU implementation uses double precision. The GPU implementation uses single precision for the principal device-side solver arrays, while retaining double precision where required by preprocessing and final diagnostics. The overlap projection, affine updates, and PSD projections are performed on the GPU. PSD projections over collections of small clique blocks are grouped by size or padded into compatible buckets and evaluated with batched eigensolvers on concurrent CUDA streams. Larger blocks use the general eigensolver path.

The experiments were conducted in Julia 1.11.7 on a machine with Intel Xeon Gold 6459C CPUs and NVIDIA GeForce RTX 4090 GPUs with 48~GiB of memory each.

Tables~\ref{tab:julia-main} and \ref{tab:julia-large-theta} report Julia CPU and GPU solve times. Loading, clique construction, conversion, and host--device preparation are excluded from the solve time. Each timed invocation receives a fresh copy of the solver inputs, and the reported value is the minimum BenchmarkTools sample with one solver evaluation per sample. These timings are intended as an internal CPU--GPU comparison of the Julia implementation. They are not directly compared with the {\sc Matlab} solver timings because the software environments, arithmetic precision, hardware, and stopping tolerances differ. The Julia experiments use the same relative KKT residual, with the tolerance specified in each table caption.

For the principal theta set, the GPU is faster on 60 of 62 instances and the median CPU-to-GPU speedup is 9.81. For the ten larger G-set-derived theta instances, the GPU version is faster in every case, with a median speedup of 5.68. For the 24 max-norm instances, the GPU solve time is lower in every displayed row. The median speedup is 4.00 over all 24 instances and approximately 2.95 over the 14 instances for which both the CPU and GPU solvers satisfy the stopping tolerance of $10^{-4}$. The few small theta instances that favor CPU execution do not provide enough work to amortize GPU launch and eigensolver setup costs.

In the tables, a superscript $\ast$ marks a run that did not achieve the strict tolerance $10^{-4}$. We observed this status varies across repeated single-precision GPU runs, i.e., the same problem is sometimes partially solved and sometimes fully. Gray shading denotes the shorter solve time in each row.

\footnotesize
\begin{longtable}{|l|r|r|r|r||r|r||r|r|}
\caption{Julia CPU and GPU results on some Lovász theta and max-norm matrix recovery problems at tolerance $10^{-4}$.}\label{tab:julia-main}\\
\hline
\multicolumn{5}{|c||}{} &
\multicolumn{2}{c||}{CPU} &
\multicolumn{2}{c|}{GPU} \\
\hline
Problem & \%nnz & $n$ & $m$ & $p$ &
Iter & Solve & Iter & Solve \\
\hline
\endfirsthead
\hline
\multicolumn{5}{|c||}{} &
\multicolumn{2}{c||}{CPU} &
\multicolumn{2}{c|}{GPU} \\
\hline
Problem & \%nnz & $n$ & $m$ & $p$ &
Iter & Solve & Iter & Solve \\
\hline
\endhead
theta4 & 11.1 & 201 & 1949 & 1 & 459 & 8.85 & 527 & \best{1.13} \\
theta6 & 10.6 & 301 & 4375 & 1 & 494 & 17.70 & 494 & \best{1.71} \\
theta8 & 10.6 & 401 & 7905 & 1 & 462 & 28.91 & 462 & \best{2.29} \\
theta10 & 10.5 & 501 & 12470 & 1 & 466 & 49.16 & 466 & \best{3.20} \\
theta12 & 10.5 & 601 & 17979 & 1 & 418 & 70.40 & 418 & \best{3.32} \\
theta42 & 31.1 & 201 & 5986 & 1 & 152 & 3.53 & 152 & \best{0.41} \\
theta62 & 30.6 & 301 & 13390 & 1 & 158 & 7.78 & 158 & \best{0.82} \\
theta82 & 30.4 & 401 & 23872 & 1 & 147 & 11.65 & 147 & \best{1.37} \\
theta83 & 50.3 & 401 & 39862 & 1 & 107 & 8.16 & 108 & \best{2.24} \\
theta102 & 30.5 & 501 & 37467 & 1 & 152 & 18.21 & 152 & \best{2.72} \\
theta103 & 50.4 & 501 & 62516 & 1 & 97 & 10.88 & 97 & \best{0.64} \\
theta104 & 70.1 & 501 & 87245 & 1 & 105 & 12.19 & 105 & \best{0.71} \\
theta123 & 50.3 & 601 & 90020 & 1 & 109 & 20.85 & 109 & \best{0.77} \\
theta162 & 40.1 & 801 & 127600 & 1 & 114 & 38.86 & 114 & \best{1.20} \\
\hline
MANN-a27 & 1.8 & 379 & 703 & 15 & 315 & \best{1.44} & 315 & 2.34 \\
MANN-a45 & 0.66 & 1036 & 1981 & 43 & 527 & \best{8.34} & 529 & 12.16 \\
brock200-1 & 26.6 & 201 & 5067 & 1 & 247 & 4.60 & 247 & \best{0.62} \\
brock200-4 & 35.2 & 201 & 6812 & 1 & 147 & 2.70 & 147 & \best{0.41} \\
brock400-1 & 25.7 & 401 & 20078 & 1 & 228 & 16.88 & 228 & \best{1.79} \\
brock800-1 & 35.3 & 801 & 112096 & 1 & 117 & 40.60 & 117 & \best{1.12} \\
brock800-4 & 35.3 & 801 & 111958 & 1 & 123 & 41.66 & 123 & \best{1.13} \\
c-fat200-1 & 92.4 & 201 & 18367 & 1 & 484 & 5.79 & 484 & \best{2.54} \\
c-fat500-10 & 62.8 & 501 & 78124 & 1 & 1159 & 60.17 & 1159 & \best{5.94} \\
hamming10-4 & 17.3 & 1025 & 89601 & 1 & 92 & 47.74 & 92 & \best{1.31} \\
johnson32-2-4 & 12.7 & 497 & 14881 & 1 & 89 & 4.08 & 83 & \best{0.67} \\
keller4 & 36.2 & 172 & 5101 & 1 & 219 & 7.20 & 219 & \best{0.49} \\
keller5 & 25.1 & 777 & 74711 & 1 & 293 & 123.27 & 294 & \best{2.12} \\
p-hat300-1 & 75.9 & 301 & 33918 & 1 & 254 & 10.97 & 254 & \best{3.28} \\
p-hat1000-1 & 75.6 & 1001 & 377248 & 1 & 220 & 108.21 & 221 & \best{2.81} \\
san200-0.7-1 & 31.0 & 201 & 5971 & 1 & 1911 & 30.84 & 1656 & \best{4.86} \\
san1000 & 50.0 & 1001 & 249001 & 1 & 142 & 42.35 & 142 & \best{1.85} \\
sanr200-0.7 & 31.3 & 201 & 6033 & 1 & 172 & 3.89 & 172 & \best{0.46} \\
G43 & 2.3 & 1001 & 9991 & 1 & 972 & 258.36 & 972 & \best{10.37} \\
G44 & 2.3 & 1001 & 9991 & 1 & 972 & 256.26 & 972 & \best{10.39} \\
G45 & 2.3 & 1001 & 9991 & 1 & 976 & 256.79 & 967 & \best{10.42} \\
G46 & 2.3 & 1001 & 9991 & 1 & 999 & 261.81 & 1015 & \best{10.65} \\
G47 & 2.3 & 1001 & 9991 & 1 & 981 & 250.67 & 974 & \best{10.45} \\
G51 & 1.5 & 1001 & 5910 & 82 & 590 & 63.01 & 596 & \best{11.89} \\
G52 & 1.5 & 1001 & 5917 & 86 & 1162 & 132.32 & 1167 & \best{23.48} \\
G53 & 1.5 & 1001 & 5915 & 86 & 951 & 111.20 & 1101 & \best{85.80} \\
G54 & 1.5 & 1001 & 5917 & 86 & 1127 & 117.18 & 1130 & \best{20.26} \\
\hline
1dc.128 & 20.0 & 129 & 1472 & 1 & 412 & 4.49 & 412 & \best{0.82} \\
1dc.512 & 8.0 & 513 & 9728 & 1 & 1036 & 139.88 & 1036 & \best{5.61} \\
1dc.1024 & 4.9 & 1025 & 24064 & 1 & 1059 & 567.37 & 1060 & \best{13.41} \\
1dc.2048 & 2.9 & 2049 & 58368 & 1 & 1432 & 2446.21 & 1532 & \best{38.04} \\
1et.128 & 10.4 & 129 & 673 & 1 & 331 & 7.31 & 332 & \best{0.63} \\
1et.256 & 6.2 & 257 & 1665 & 10 & 343 & 1.74 & 344 & \best{1.24} \\
1et.512 & 3.6 & 513 & 4033 & 26 & 315 & 4.34 & 317 & \best{1.50} \\
1et.1024 & 2.1 & 1025 & 9601 & 61 & 594 & 26.67 & 580 & \best{6.11} \\
1et.2048 & 1.2 & 2049 & 22529 & 144 & 664 & 105.74 & 664 & \best{23.42} \\
1tc.128 & 8.5 & 129 & 513 & 1 & 1169 & 14.40 & 1169 & \best{2.15} \\
1tc.256 & 5.1 & 257 & 1313 & 7 & 494 & 1.94 & 470 & \best{1.39} \\
1tc.512 & 3.1 & 513 & 3265 & 21 & 564 & 7.36 & 551 & \best{2.65} \\
1tc.1024 & 1.8 & 1025 & 7937 & 57 & 597 & 22.27 & 617 & \best{5.48} \\
1tc.2048 & 1.0 & 2049 & 18945 & 130 & 504 & 54.08 & 525 & \best{12.16} \\
1zc.128 & 15.8 & 129 & 1121 & 1 & 400 & 8.79 & 366 & \best{0.70} \\
1zc.512 & 5.8 & 513 & 6913 & 1 & 817 & 193.39 & 817 & \best{4.04} \\
1zc.1024 & 3.5 & 1025 & 16641 & 1 & 1217 & 855.87 & 1222 & \best{13.06} \\
1zc.2048 & 2.0 & 2049 & 39425 & 1 & 1483 & 4059.58 & 1483 & \best{47.45} \\
2dc.512 & 42.3 & 513 & 54896 & 1 & 469 & 58.89 & 503 & \best{12.98} \\
2dc.1024 & 32.5 & 1025 & 169163 & 1 & 412 & 211.45 & 430 & \best{4.62} \\
2dc.2048 & 24.2 & 2049 & 504452 & 1 & 387 & 645.36 & 454 & \best{13.51} \\
\hline
rnd(300,100)\_0.1 & 4.0 & 400 & 3399 & 6 & 833 & 14.34 & 832 & \best{5.69} \\
rnd(300,100)\_0.15 & 5.9 & 400 & 4899 & 6 & 578 & 9.69 & 579 & \best{3.95} \\
rnd(300,100)\_0.2 & 7.5 & 400 & 6399 & 6 & 514 & 8.68 & 529 & \best{3.61} \\
rnd(300,100)\_0.25 & 9.6 & 400 & 7899 & 6 & 449 & 7.54 & 501 & \best{3.26} \\
rnd(300,100)\_0.3 & 11.5 & 400 & 9399 & 6 & 441 & 7.91 & 381 & \best{2.86} \\
rnd(300,100)\_0.35 & 13.4 & 400 & 10899 & 6 & 440 & 7.86 & 372 & \best{2.72} \\
rnd(500,200)\_0.1 & 4.2 & 700 & 10699 & 5 & 940 & 56.01 & 943 & \best{12.47} \\
rnd(500,200)\_0.15 & 6.3 & 700 & 15699 & 5 & 824 & 48.51 & 817 & \best{11.49} \\
rnd(500,200)\_0.2 & 8.3 & 700 & 20699 & 5 & 851 & 51.96 & 817 & \best{12.25} \\
rnd(500,200)\_0.25 & 10.3 & 700 & 25699 & 5 & 609 & 40.09 & 491 & \best{8.64} \\
rnd(500,200)\_0.3 & 12.4 & 700 & 30699 & 5 & 500 & 32.77 & 480 & \best{10.87} \\
rnd(500,200)\_0.35 & 14.2 & 700 & 35699 & 5 & 369 & 23.37 & 389 & \best{11.07} \\
rnd(1200,200)\_0.1 & 2.5 & 1400 & 25399 & 12 & 1401 & 214.57 & 1516 & \best{53.41$^{\ast}$} \\
rnd(1200,200)\_0.15 & 3.7 & 1400 & 37399 & 12 & 1151 & 165.59 & 1132 & \best{42.92$^{\ast}$} \\
rnd(1200,200)\_0.2 & 5.0 & 1400 & 49399 & 12 & 1203 & 184.67 & 1151 & \best{56.99$^{\ast}$} \\
rnd(1200,200)\_0.25 & 6.2 & 1400 & 61399 & 12 & 585 & 84.58 & 636 & \best{18.34$^{\ast}$} \\
rnd(1200,200)\_0.3 & 7.4 & 1400 & 73399 & 12 & 480 & 69.59 & 547 & \best{15.75$^{\ast}$} \\
rnd(1200,200)\_0.35 & 8.6 & 1400 & 85399 & 12 & 480 & 72.65 & 688 & \best{17.81} \\
rnd(1200,300)\_0.1 & 3.3 & 1500 & 37499 & 8 & 1451 & 315.92 & 1551 & \best{76.02$^{\ast}$} \\
rnd(1200,300)\_0.15 & 4.9 & 1500 & 55499 & 8 & 1201 & 260.65 & 1378 & \best{65.66$^{\ast}$} \\
rnd(1200,300)\_0.2 & 6.5 & 1500 & 73499 & 8 & 1251 & 279.86 & 1245 & \best{36.89} \\
rnd(1200,300)\_0.25 & 8.1 & 1500 & 91499 & 8 & 507 & 108.10 & 576 & \best{15.91$^{\ast}$} \\
rnd(1200,300)\_0.3 & 9.7 & 1500 & 109499 & 8 & 501 & 109.86 & 606 & \best{18.01$^{\ast}$} \\
rnd(1200,300)\_0.35 & 11.3 & 1500 & 127499 & 8 & 579 & 132.63 & 656 & \best{23.82$^{\ast}$} \\
\hline
\end{longtable}

\begin{longtable}{|l|r|r|r|r||r|r||r|r|}
\caption{Julia CPU and GPU results on some Lovász theta problems at tolerance $10^{-3}$.}\label{tab:julia-large-theta}\\
\hline
\multicolumn{5}{|c||}{} &
\multicolumn{2}{c||}{CPU} &
\multicolumn{2}{c|}{GPU} \\
\hline
Problem & \%nnz & $n$ & $m$ & $p$ &
Iter & Solve & Iter & Solve \\
\hline
\endfirsthead
\hline
\multicolumn{5}{|c||}{} &
\multicolumn{2}{c||}{CPU} &
\multicolumn{2}{c|}{GPU} \\
\hline
Problem & \%nnz & $n$ & $m$ & $p$ &
Iter & Solve & Iter & Solve \\
\hline
\endhead
G55 & 0.16 & 5001 & 12499 & 358 & 372 & 524.28 & 364 & \best{125.59} \\
G56 & 0.16 & 5001 & 12499 & 358 & 372 & 516.56 & 370 & \best{134.83} \\
G57 & 0.14 & 5001 & 10001 & 197 & 367 & 63.32 & 394 & \best{8.72} \\
G62 & 0.10 & 7001 & 14001 & 280 & 439 & 121.12 & 406 & \best{12.75} \\
G65 & 0.09 & 8001 & 16001 & 340 & 369 & 132.94 & 379 & \best{18.82} \\
G66 & 0.08 & 9001 & 18001 & 396 & 396 & 169.04 & 406 & \best{25.95} \\
G67 & 0.07 & 10001 & 20001 & 438 & 436 & 213.38 & 441 & \best{30.39} \\
G58 & 0.30 & 5001 & 29571 & 485 & 524 & 1176.83 & 589 & \best{320.12} \\
G60 & 0.11 & 7001 & 17149 & 519 & 448 & 1379.93 & 461 & \best{288.29} \\
G63 & 0.21 & 7001 & 41460 & 684 & 554 & 2364.95 & 655 & \best{487.35$^{\ast}$} \\
\hline
\end{longtable}
\noindent Here, $n$ is the dimension of the undecomposed SDP matrix after theta conversion (with $n=d_1+d_2$ for max-norm recovery), while $m$ is the number of original affine equality constraints. A light-gray cell identifies the lower CPU/GPU solve time for that row. The superscript $\ast$ has the meaning stated in Section~\ref{sec:julia-gpu}.

\normalsize
\section{Summary} \label{sec-summary} We proposed {\sc McADMM}, an enhanced first-order SDP solver capable of handling multi-block structured SDPs, particularly those with bound constraints that typically arise from decompositions of sparse large-scale problems. Importantly, the structures of our algorithms are applicable to multi-block SDPs based on \textit{any} type of decomposition, including non-chordal (non-exact) ones. In this regard, \decompadmm\;widens the horizon for decomposition-based first-order algorithms, since the latter is an effective and widely used approach for handling large-scale and constraint-heavy SDPs.

We demonstrated the key differences of our solver compared to other decomposition based algorithms. Our algorithm handles the consistency constraints via a more efficient \textit{linear-space projection} update so that as a consequence, our clique assignment algorithm CMD may focus on {reducing the estimated PSD projection cost} instead of minimizing the overlaps. Based on our numerical results, the combination of these two features allows McADMM to surpass strong decomposition-based solvers such as CDCS which uses the well-established SparseCoLO package for chordal decomposition. The Julia experiments further show that the multi-clique structure can be exploited efficiently on GPUs, with the largest benefits appearing once the PSD projection workload is sufficient to amortize device-side overhead.

Future work may include various generalizations of our solver such as the implementation of acceleration techniques on the updating schemes in Algorithms~\ref{alg:mcadmm-general} and~\ref{alg:mcadmm-bounded}. Possibly a more concrete avenue involves employing a stronger optimizer -- such as the semismooth Newton-CG based solver -- described in Appendix~\ref{sec-ALM} to process the subproblems. One can also consider broader applications of \decompadmm\;on other large-scale problems such as graph matching problems or point cloud registrations.

\begin{appendix}

\section{Conversion of problems with dense aggregate sparsity patterns} \label{sec-theta} Consider an SDP with sparse constraint matrices $A_1, \ldots , A_m$ but a dense cost matrix i.e., the aggregate sparsity pattern of the problem admits a satisfactory clique decomposition if not for the inclusion of the pattern in $C$. Another assumption is that $C$ differs from a sparse matrix only by a low-rank dense matrix:
\begin{align}\label{eq.29}
    C = C^s - VDV^\top,
\end{align}
where $C^s$ is sparse,  $V\in \R^{n\times r}$ is dense and low-rank, and $D \in \S^r$ is positive definite with $r \ll n$. Although these assumptions may seem very specific, they are applicable to a few important applications, e.g., problems with few constraints and $C$ that is itself negative semidefinite and low-rank.

For SDPs with these properties, we may exploit a hidden sparse structure in an indirect way. Note that the PSD constraint in the dual SDP (of \eqref{eq.1}) can generally be written as $S = C -  \cA^* y - Z \in \S^{n}_+$. Using \eqref{eq.29} and the Schur's complement lemma, we can represent this as
\begin{align}\label{eq.30}
    C^s - \cA^* y -Z- VDV^\top \in \S^n_+\;\Leftrightarrow\; \;
    \left[ \begin{array}{cc}
       C^s - \cA^* y-Z & V \\[5pt] V^\top & D^{-1}
    \end{array}\right]  \in \S^{n+r}_+.
\end{align}

With this, the dual SDP is equivalent to
\begin{align}\label{eq.31}
    \max_{y,Z}		\hspace{0.1cm} & \inprod{b}{y} - \delta^*_{\N^{n}} (-Z)\nonumber\\[5pt]
    \text{s.t.} \hspace{0.1cm} & \underbrace{\left[ \begin{array}{cc}
    C^s  & V \\[5pt] V^\top & D^{-1}
    \end{array}\right]}_{\widehat{C}}
    - \sum_{i=1}^m y_i \underbrace{\left[ \begin{array}{cc}
    A_i  & 0 \\[5pt] 0 & 0
    \end{array}\right]}_{\widehat{A}_i}
    -  \underbrace{\left[ \begin{array}{cc}
    Z  & 0 \\[5pt] 0 & 0
    \end{array}\right]}_{\widehat{Z}}
    \in \S^{n+r}_+ ,
\end{align}
where $\widehat{C}, \widehat{A}_i$ and $\widehat{Z}$ are extended data matrices and variables in $\S^{n+r}$.

The aggregate sparsity pattern of this extended problem is then given by
\begin{eqnarray*}
    \left[ \begin{array}{cc}
    |C^s| + \sum_{i=1}^{m} |A_i|   & |V| \\[5pt] |V|^\top & |D^{-1}|
    \end{array}\right],
\end{eqnarray*}
together with the aggregate pattern $M^{\text{bound}}$ from the bounds $L$ and $U$. Clearly, this sparsity pattern is much more conducive for applications in a chordal conversion with the source of density having been transferred to $r \ll n$ rows and columns in the extended matrix.

The solution of the original SDP (\eqref{eq.1} with $C = C^s - VDV^\top$) can be deduced from that of \eqref{eq.31}, as formally stated in the following proposition.

\begin{proposition}
Suppose $\hX$ and $(\hy,\hZ)$ {form a primal-dual KKT pair} for the extended primal and dual SDPs (\eqref{eq.31} and its dual). Write $\hX = \left[ \begin{array}{cc}    \hX_{11}  & \hX_{12} 
\\[5pt] \hX_{12}^\top & \hX_{22} \end{array}\right]$ and $\hZ = \left[ \begin{array}{cc}    \hZ_{11}  & \mathbf{0} 
\\[5pt] \mathbf{0}^\top & 0 \end{array}\right]$
with $\hX_{11}, \hZ_{11} \in \S^n$ and $\hX_{12} \in \R^{n \times r}$. Then $\hX_{11}$ solves the original SDP and $(\hy,\hZ_{11})$ solves its corresponding dual together with $S = C -  \cA^* {\hy} - \hZ_{11} \in \S^n_+$, achieving the same optimal value as the extended problems.
\end{proposition}
\begin{proof}
From the complementarity condition of the extended problem:
\begin{align*}
  \left[ \begin{array}{cc}
   C^s-\cA^*\hy-\hZ_{11}  & V \\[5pt] V^\top & D^{-1}
\end{array}\right]
\left[ \begin{array}{cc}
   \hX_{11}  & \hX_{12} \\[5pt] \hX_{12}^\top & \hX_{22}
\end{array}\right]
= 0,
\end{align*}
we have
\begin{align*}
 (C^s -\cA^*\hy-\hZ_{11})\hX_{11} + V \hX_{12}^\top = 0,\quad
V^\top \hX_{11} + D^{-1}\hX_{12}^\top = 0,\quad V^\top\hX_{12} + D^{-1} \hX_{22} =0.
\end{align*}
Simplifying, we obtain
\begin{align*}
 \hX_{12}^\top = - DV^\top\hX_{11}, \quad \hX_{22} = - D V^\top \hX_{12} =DV^\top \hX_{11} V D.
\end{align*}
Thus, the complementarity condition of the original problem is recovered:
\begin{align*}
0 = (C^s-\cA^*\hy -\hZ_{11} - VDV^\top)\hX_{11} = (C-\cA^*\hy-\hZ_{11})\hX_{11}.
\end{align*}

Furthermore,
\begin{align*}
 \inprod{\widehat{C}}{\hX}  = \inprod{C^s}{\hX_{11}} + 2\inprod{V}{\hX_{12}} + \inprod{D^{-1}}{\hX_{22}} = \inprod{C^s- VDV^\top}{\hX_{11}} = \inprod{C}{\hX_{11}},
\end{align*}
whence it is also straightforward to conclude that $\inprod{A_i}{\hX_{11}} = \inprod{\widehat{A}_i}{\hX} =b_i$ for $i = 1,\ldots,m$. {Feasibility of the extended primal problem also gives $\hX_{11}\in\N^n$.} Moreover, $\hX_{11}$ and $C -  \cA^* {\hy} - \hZ_{11} \in \S^n_+$ are positive semidefinite. This implies that $\hX_{11}$ and $(\hy,\hZ_{11})$ are optimal for the original problems with equal optimal value as claimed.
\end{proof}

One prominent application of this technique is on the Lov{\'a}sz-theta SDP arising from the maximum stable set problem \cite{Lovasz1979}:
\begin{align}\label{eq.32}
\min_{Y } \hspace{0.1cm} & \inprod{-ee^\top}{Y}  \nonumber \\
\text{s.t.}     & \hspace{0.1cm}  \inprod{E^{ij}}{Y} =0  \hspace{0.2cm} \text{for } (i,j) \in \Theta \nonumber\\
		& \hspace{0.1cm}  \inprod{I}{Y} = 1, \quad Y \in \S^{n}_+.
\end{align}
Note that $C=-ee^\top$ is a rank-1 dense matrix of all minus ones, whereas depending on the edges set $\Theta$, the constraint matrices may be very sparse. Hence, it is likely very useful to apply an extension on this problem:
\begin{align}\label{eq.33}
 \min_{\hX } \hspace{0.1cm}	& \inprod {\widehat{C}} {\widehat{X}}\nonumber\\[5pt]
\mbox{s.t.} & \hspace{0.1cm}
\inprod{\widehat{A}_{l(ij)}}{\widehat{X}} =0 \hspace{0.2cm} \text{for } (i,j) \in \Theta \nonumber \\[5pt]
& \;\inprod{\widehat{A}_{|\Theta|+1}}{\widehat{X}} = 1, \quad \widehat{X} \in \S^{n+1}_+,
\end{align}
where once again, matrices such as $\widehat{X}$ are in $\mathbb{S}^{(n+1)}$. Furthermore,
$$
\widehat{C} =\begin{pmatrix} C^s &\mathbf{1} \\  \mathbf{1}^\top & 1  \end{pmatrix}, \quad
\widehat{A}_{l(ij)} = \begin{pmatrix} E^{ij} &\mathbf{0} \\  \mathbf{0}^\top & 0  \end{pmatrix}, \quad  (i,j) \in \Theta,\quad
\widehat{A}_{|\Theta|+1} = \begin{pmatrix} I_n &\mathbf{0} \\  \mathbf{0}^\top & 0  \end{pmatrix},
$$
where $C^s = 0_{n\times n}$, and $l(ij)$ denotes a linear enumeration of the edges $(ij)$ in $\Theta.$

\section{A semismooth Newton-CG based inexact proximal ADMM for solving \texorpdfstring{\eqref{eq.6}}{(6)}} \label{sec-ALM} Here, we give a brief overview of an alternative inexact proximal ADMM framework for solving \eqref{eq.6} that employs a semismooth Newton-CG (SNCG) method to solve the subproblems in each iteration. The resulting method is akin to the majorized semismooth Newton-CG augmented Lagrangian method designed in \cite{YST2015} for solving semidefinite programming problems with nonnegative constraints. However, the advantage of designing our method within the inexact proximal ADMM framework is that it can immediately enjoy convergence results already established in \cite{CST2015}. 

The template of the semismooth Newton-CG-based inexact proximal ADMM (which we call \SNCGadmm\!\!) is given as follows.
\begin{description}
\item[Template \SNCGadmm.] Fix a nonnegative summable sequence $\{ \varepsilon_k\}$. Given an initial iterate $(y^0,\xi^0,S^0,Z^0,X^0)$, perform the following steps in each iteration.
\item[Step 1] Let ${R}_{y,S}^k =  \cL^* \xi^k + Z^k - \mathscr{C} + \sig^{-1} X^k$.
Compute
\begin{eqnarray*}
\hspace{-0.5cm}
  (y^{k+1},S^{k+1}) &\approx&  \mbox{argmin}_{y,S} \; L_\sig (y,\xi^k,S,Z^k;X^k)
\\[3pt]
&=&\mbox{argmin}_{y,S} \Big\{  \inprod{-b}{y} + \mathscr{F}^*(-S) + \frac{\sig}{2}
\norm{\mathscr{A}^* y +S + {R}_{y,S}^k}^2
\Big\}
\end{eqnarray*}
such that the residual to the optimality condition
\begin{eqnarray}
 \delta^{k+1}_{y,S} \in \partial_{(y,S)} L_\sig (y^{k+1},\xi^k,S^{k+1},Z^k;X^k)
 \label{eq.67}
\end{eqnarray}
satisfies $\norm{\delta^{k+1}_{y,S}} \leq \varepsilon_k$.
\item[Step 2] Let
${R}_{\xi,Z}^k = \mathscr{A}^* y^{k+1} +  S^{k+1}  - \mathscr{C} + \sig^{-1} X^k.$
Compute
\begin{eqnarray*}
\hspace{-0.9cm}
 ({\xi}^{k+1},Z^{k+1}) &\approx& \mbox{argmin}_{\xi,Z} \;
 L_\sig (y^{k+1},\xi,S^{k+1},Z;X^k)
\\[3pt]
&=&\mbox{argmin}_{\xi,Z} \Big\{  \inprod{-h}{\xi} + \mathscr{G}^*(-Z) + \frac{\sig}{2}
\norm{\cL^* \xi + Z  + {R}_{\xi,Z}^k}^2
\Big\}
\end{eqnarray*}
such that the residual to the optimality condition
\begin{eqnarray}
 \delta^{k+1}_{\xi,Z} \in \partial_{(\xi,Z)} L_\sig (y^{k+1},\xi^{k+1},S^{k+1},Z^{k+1};X^k)
 \label{eq.68}
\end{eqnarray}
satisfies $\norm{\delta^{k+1}_{\xi,Z}} \leq \varepsilon_k$.
\item[Step 3.] Compute
\begin{eqnarray*}
 X^{k+1} = X^k + \tau \sig (\mathscr{A}^* y^{k+1} + \cL^* \xi^{k+1} + S^{k+1} + Z^{k+1} - \mathscr{C}),
\end{eqnarray*}
where $\tau \in [1,(1+\sqrt{5})/2)$ is the steplength which as usual is chosen to be
$1.618$.
\end{description}

\medskip
The SNCG method is used to efficiently solve the combined subproblems in Steps 1 and 2. For the former, let $W^k(y)  = \mathscr{A}^* y + {R}_{y,S}^k$. By noting that
\begin{eqnarray*}
&& \hspace{-0.7cm}   \mbox{min}_S \Big\{  \sig^{-1} \mathscr{F}^*(-S)  + \frac{1}{2}
\norm{S +W^k(y)}^2
\Big\} = M_{\sig^{-1} \mathscr{F}^*}(W^k(y)),
\end{eqnarray*}
we can solve the subproblem in Step 1 as follows: compute
\begin{eqnarray}
  y^{k+1} &\approx &  \mbox{argmin} \Big\{
\theta^k(y) := \inprod{-b}{y}  + {\sig}M _{\sig^{-1} \mathscr{F}^*}(W^k(y))
\Big\}
\label{eq.69}
\\[5pt]
S^{k+1} & = & -\mbox{Prox}_{\sig^{-1} \mathscr{F}^*} (\mathscr{A}^* y^{k+1} + {R}^k_{y,S}).
\label{eq.70}
\end{eqnarray}
Note that \eqref{eq.69} is an unconstrained minimization problem and its optimality condition is given by:
\begin{eqnarray}
 0= \grad \theta^k (y) = -b + \mathscr{A} {\rm Prox}_{\sig \mathscr{F}}(\sig W^k(y)) .
\label{eq.71}
\end{eqnarray}
In \cite{ZSToh2010,WST2010}, the design of the SNCG for solving a problem of the form \eqref{eq.69} based on the nonsmooth equation \eqref{eq.71} has already been described in detail.

In order to solve the subproblem in Step 1 within the inexactness criterion, we can estimate how accurately \eqref{eq.69} must be solved. The optimality conditions for the subproblem in Step 1 are
\begin{eqnarray}
  0 &=& -b + \sig \mathscr{A} (\mathscr{A}^* y + S + R^k_{y,S})
\label{eq.72}
\\[5pt]
  0 & \in & -\partial \mathscr{F}^*(-S) + \sig (\mathscr{A}^*y + S + R^k_{y,S}).
\label{eq.73}
\end{eqnarray}
From the choice of $S^{k+1}$ in \eqref{eq.70}, one can show that it satisfies the condition \eqref{eq.73} with $y=y^{k+1}$, $S = S^{k+1}$. Now
\begin{eqnarray*}
  \mathscr{A}^* y^{k+1}  + S^{k+1} + R^k_{y,S}
=  W^k(y^{k+1}) - {\rm Prox}_{\sig^{-1}\mathscr{F}^*}(W^k(y^{k+1}))
= \sig^{-1} {\rm Prox}_{\sig \mathscr{F}}(\sig W^k(y^{k+1})).
\end{eqnarray*}
Thus the residual to the conditions \eqref{eq.72}--\eqref{eq.73} is given by
\begin{eqnarray*}
 -b + \mathscr{A} {\rm Prox}_{\sig \mathscr{F}}(\sig W^k(y^{k+1}))
&=  &
-b + \sig \mathscr{A} (\mathscr{A}^* y^{k+1} + S^{k+1} + R^k_{y,S})
\\[5pt]
0 & \in & -\partial \mathscr{F}^*(-S^{k+1}) + \sig (\mathscr{A}^*y^{k+1} + S^{k+1} + R^k_{y,S}).
\end{eqnarray*}
Hence the approximate solution $y^{k+1}$ of \eqref{eq.69} must satisfy the accuracy condition that
$$
  \norm{ b - \mathscr{A} {\rm Prox}_{\sig \mathscr{F}}(\sig W^k(y^{k+1}))} \leq \varepsilon_k.
$$

We should mention that much flexibility is allowed in the design of our algorithm for solving \eqref{eq.6}. For example, in Step 1 of the \SNCGadmm\!\!\!, one may replace it by Steps \textbf{2a-2c} of \SGSadmm\; (in the manner of Algorithm~\ref{alg:mcadmm-general}) without affecting the convergence of the algorithm. The decision of whether to make such a replacement is problem dependent. Typically Step 1 of \SNCGadmm\;should be used if superlinear convergence is observed in the SNCG method for solving \eqref{eq.69}, otherwise it is more economical to replace it by the three \SGSadmm\!\! steps since an SNCG subproblem can be significantly more expensive.

As usual, all the above discussion for Step 1 is also applicable to Step 2, considering $(\xi, Z)$ instead of $(y,S)$.

\end{appendix}

\end{document}